\documentclass{amsart}
\usepackage{amsthm,amsfonts,amsmath,amscd,amssymb,latexsym,epsfig,epic}
\newcommand{\gl}{{\mathfrak g \mathfrak l}}

\newcommand{\cx}{{\mathbb C}}
\newcommand{\diag}{\operatorname{diag}}

\newcommand{\Res}{\operatorname{Res}}

\newcommand{\Hom}{\operatorname{Hom}}

\newcommand{\Ker}{\operatorname{Ker}}
\newcommand{\Coker}{\operatorname{Coker}}

\newcommand{\Jac}{\operatorname{Jac}}
\newcommand{\ord}{\operatorname{ord}}
\newcommand{\Tot}{\operatorname{Tot}}

\numberwithin{equation}{section}

\newtheorem{theorem}{Theorem}[section]

\newtheorem{lemma}[theorem]{Lemma}

\newtheorem{corollary}[theorem]{Corollary}

\newtheorem{proposition}[theorem]{Proposition}

\theoremstyle{remark}

\newtheorem{remark}[theorem]{Remark}

\newtheorem{definition}[theorem]{Definition}

\newtheorem{example}[theorem]{Example}

\newcommand{\oC}{{\mathbb{C}}}

\newcommand{\oP}{{\mathbb{P}}}

\newcommand{\oR}{{\mathbb{R}}}

\newcommand{\oT}{{\mathbb{T}}}

\newcommand{\oZ}{{\mathbb{Z}}}

\newcommand{\sD}{{\mathcal{D}}}

\newcommand{\sG}{{\mathcal{G}}}   

\newcommand{\sO}{{\mathcal{O}}}

\newcommand{\sS}{{\mathcal{S}}}

\newcommand{\fH}{{\mathfrak{h}}}

\begin{document}

\title[Line bundles on  spectral curves and the GLT]{Line bundles on spectral curves and the generalised Legendre transform construction of hyperk\"ahler metrics}
\author{Roger Bielawski}
\address{School of Mathematics\\
University of Leeds\\Leeds LS2 9JT\\ UK}



\begin{abstract} An analogue of the correspondence between $GL(k)$-conjugacy classes of matricial polynomials and line bundles is given for $K$-conjugacy classes, where $K\subset GL(k)$ is one of the following: maximal parabolic, maximal torus, $GL(k-1)$ embedded diagonally. The generalised Legendre transform construction of hyperk\"ahler metrics is studied further, showing that many known hyperk\"ahler metrics (including the ones on coadjoint orbits) arise in this way, and giving a large class of new (pseudo-)hyperk\"ahler metrics, analogous to monopole metrics.
\end{abstract}

\maketitle
\thispagestyle{empty}
\section{Introduction}

A fundamental result in the theory of integrable systems is the correspondence between matrix Lax equations with a spectral parameter and linear flows on the affine Jacobian of a spectral curve. In its simplest version, this result can be stated as in \cite{Beau}: there is a 1-1 correspondence between the affine Jacobian $\Jac S-\Theta$ of a smooth compact curve $S\in  |\sO(kd)|$ of degree $k$ and $GL(k,\cx)$-conjugacy classes of $\gl(k,\cx)$-valued polynomials $A(\zeta)=\sum_{i=0}^d A_i\zeta^i$, the spectrum of which is $S$.\par
Thus, up to conjugation, a matricial polynomial can be recovered from an algebro-geometric data associated to its spectrum. Our first aim in the present work is to recover more of the matricial polynomial than just ``up to conjugation". We show in \S3 how one can recover conjugacy classes of $A(\zeta)$ with respect to proper subgroups of $GL(k,\cx)$, particularly with respect to a maximal parabolic $P$ and a maximal torus $T$. The conjugacy classes with respect to $P$ correspond to divisors on $S$, rather than to line bundles, while the conjugacy classes with respect to $T$ correspond to a  sequence of curves $S_1,S_2,\dots, S_k=S$ with $S_i\in |\sO(id)|$ satisfying additional conditions (see Theorem \ref{O(2)}). The curves $S_i$ are given by the {\em Gelfand-Zeitlin map} \cite{GKL,KW,Bie-Pidst} with a spectral parameter.
\par
The $d=2$ case of the above setting is closely related to hyperk\"ahler geometry and, in particular, to (pseudo-)hyperk\"ahler metrics which can be obtained by means of the generalised Legendre transform (GLT) of Lindstr\"om and Ro\v{c}ek \cite{LR} (see sections \ref{hk} and \ref{glt} for a review of hyperk\"ahler metrics and of the GLT). This class includes all toric hyperk\"ahler manifolds \cite{BD}, the natural metrics on the moduli spaces of $SU(2)$-monopoles \cite{IR,Hough}, and the gravitational instantons of types $A_k$ \cite{DM} and $D_k$ \cite{CRW, CK, CH}. 
\par
 From one point of view \cite{TQ, DM}, the metrics constructed via GLT are those with a generalised symmetry of maximal rank, i.e. the twistor projection $Z^{2n+1}\rightarrow \oP^1$ factorises via a vector bundle $Z^{2n+1}\rightarrow E\rightarrow \oP^1$ of rank $n$, with the fibres of  $Z^{2n+1}\rightarrow E$ being Lagrangian for the twisted symplectic form of $Z$. The bundle $E$ splits as $\bigoplus_{j=1}^n \sO(2r_j)$ and $r_j=1$ yields a genuine symmetry of the hyperk\"ahler structure. The hyperk\"ahler structure is then recovered from a single function $F$  on the space $V$ of (real) sections of $E$. The hyperk\"ahler manifold $M$ is  a torus (or another abelian group) bundle over a submanifold $X\subset V$, with the dimension of the torus equal to $\#\{j;r_j=1\}$. The submanifold $X$ is the image of the generalised moment map on $M$.
The function $F$ can be obtained as a contour integral  of a holomorphic function $G$ of $n+1$ variables (or a sum of such).
The functions $G$ and $F$ are often found by ad hoc methods, depending on the example, and it is one of the aims of this paper to present a formal and systematic derivation of $F$ for a class of hyperk\"ahler manifolds obtained via GLT.
\par  
 The holomorphic function $G$ can be singular or multi-valued, and in the present paper we consider the case of GLT, where the function $G$ arises from an element of $H^1(D,\sO)$ for some branched covering $D$ of $E=\bigoplus_{j=1}^n \sO(2r_j)$. For example, when $E=\sO(2)\oplus \sO(4)$ we can consider a $2$-fold covering: $$D_1=\{(\eta,\alpha_1,\alpha_2)\in \sO(2)\oplus \sO(2)\oplus \sO(4);\, \eta^2+\alpha_1\eta+\alpha_2=0\},$$ or a $3$-fold one: $$D_2=\{(\eta,\alpha_1,\alpha_2)\in \sO(2)\oplus \sO(2)\oplus \sO(4);\, (\eta+\alpha_1)(\eta^2+\alpha_2)=0\}.$$ The space $V$ should be now viewed as a space of spectral curves - all compact curves in $|\sO(4)|$ for $D_1$, and a subset of the set of reducible compact curves in $|\sO(6)|$ for $D_2$.  Our first observation is that, if $D$ is chosen so that $V$ corresponds to {\em all} compact reducible curves $S$  of the form $S=S_1\cup\dots\cup S_k$,  with components $S_i\in |\sO(2m_i)|$, $i=1,\dots,k$, then the submanifold $X$ of $V$ (the image of the generalised moment map on $M$) corresponds to curves $S$ on which a certain  line bundle is trivial (see \S\ref{glt_spec}). We  then restrict ourselves further, to hyperk\"ahler metrics, the twistor space of which can be trivialised using spectral curves and  sections of line bundles. This is the case for $SU(2)$-monopole metrics and for ($A_k$- and $D_k$-) gravitational instantons, and we show that any such hyperk\"ahler manifold can be constructed via GLT. Moreover, we compute explicitly the function $F$ for such a manifold. The examples include $SU(N)$-monopole metrics, asymptotic monopole metrics considered in \cite{clusters}, and, somewhat surprisingly, hyperk\"ahler metrics on regular adjoint orbits of $GL(k,\cx)$. In this last example, the work done in \S3 (particularly Theorem \ref{O(2)}) plays a crucial role.
\par
Conversely, given a space $V$ of spectral curves and a line bundle $K$ on $T\oP^1$ with $c_1(K)=0$, we write down a function $F:V\rightarrow \oR$, the GLT of which produces a pseudo-hyperk\"ahler metric with twistor space trivialised using spectral curves and  sections of $K$. This gives a huge family of hyperk\"ahler metrics, analogous to monopole metrics. We observe, for example, that there is a ``master metric" from which all (pseudo-)hyperk\"ahler metrics on regular adjoint orbits of $GL(k,\cx)$ can be obtained as twistor quotients. Moreover, just as monopole metrics correspond to Nahm's equations, these metrics correspond to other integrable systems, defined by ODEs on triples of matrices. Essentially, $K$ can be determined by a harmonic polynomial on $\oR^3$, with Nahm's equations corresponding to a quadratic polynomial. In the last section, we consider briefly the matrix-valued ODEs corresponding to a cubic harmonic polynomial.

\section{Line bundles and matricial polynomials}

We give here (mostly following \cite{Theta}) a brief summary of spectral curves and line bundles, including Beauville's theorem. 

In what follows, $\oT$ denotes the total space of the line bundle $\sO(d)$ on $\oP^1$, $\pi:\oT\rightarrow \oP^1$ is the
projection, $\zeta$ is the affine coordinate on $\oP^1$ and $\eta$ is the fibre coordinate on $\oT$. In other words $\oT$ is obtained by gluing
two copies of $\oC^2$ with coordinates $(\zeta,\eta)$ and $(\tilde{\zeta},\tilde{\eta})$ via:
$$ \tilde{\zeta}=\zeta^{-1}, \quad \tilde{\eta}=\eta/\zeta^d.$$
We denote the corresponding two open subsets of $\oT$ by $U_0$ and $U_\infty$.

Let $S$ be a compact algebraic curve in the linear system $\sO(dk)$, i.e. over $\zeta\neq \infty$ $S$ is defined by the equation
\begin{equation} P(\zeta,\eta)= \eta^k+a_1(\zeta)\eta^{k-1}+\cdots +a_{k-1}(\zeta)\eta+ a_k(\zeta)=0,\label{S}\end{equation}
where  $a_i(\zeta)$ is a polynomial of degree $di$. $S$ can be singular or non-reduced.
\par
We recall the following facts (see, e.g., \cite{AHH}):
\begin{proposition} The group $H^1(\oT,\sO_\oT)$ (i.e. line bundles on $\oT$ with zero first Chern class) is generated by  $\eta^i\zeta^{-j}$, $i>0$, $0<j<2i$. The corresponding line bundles have transition functions $\exp(\eta^i\zeta^{-j})$ from $U_0$ to $U_\infty$.\hfill $\Box$\label{T}\end{proposition}
\begin{proposition} The natural map $H^1(\oT,\sO_\oT)\rightarrow H^1(S,\sO_S)$ is a surjection, i.e. $H^1(S,\sO_S)$ is generated by $\eta^i\zeta^{-j}$, $0<i\leq k-1$, $0<j<id$.\hfill $\Box$\label{all}\end{proposition}

Thus, the (arithmetic) genus of $S$ is $g=(k-1)(dk-2)/2$.

For a smooth $S$, the last proposition describes line bundles of degree $0$ on $S$.
In general, by a line bundle we mean an invertible sheaf. Its degree is defined  as its Euler characteristic plus $g-1$. The theta divisor $\Theta$ is  the set of line bundles  of degree $g-1$ which have a non-zero section.
\par
Let $\sO_\oT(i)$ denote the pull-back of $\sO(i)$ to $\oT$ via $\pi:\oT\rightarrow \oP^1$. If $E$ is a sheaf on $\oT$ we denote by $E(i)$ the sheaf
$E\otimes \sO_\oT(i)$ and similarly for sheaves on $S$. In particular, $\pi^\ast \sO$ is identified with $\sO_S$. We note that the canonical bundle $K_S$ is isomorphic to $\sO_S(d(k-1)-2)$.
\par
If $F$ is a line bundle of degree $0$ on $S$, determined by a cocycle $q\in H^1(\oT,\sO_\oT)$, and $s\in H^0\bigl(S, F(i)\bigr)$, then we denote by
$s_0,s_\infty$ the representation of $s$ in the trivialisation $U_0,U_\infty$, i.e.:
\begin{equation} s_\infty(\zeta,\eta)=\frac{e^q}{\zeta^i}s_0(\zeta, \eta).\label{represent}\end{equation}

 We recall the following theorem of Beauville \cite{Beau}:
\begin{theorem} There is a $1-1$ correspondence between the affine Jacobian $J^{g-1}-\Theta$ of line bundles of degree $g-1$ on $S$ and $GL(k,\cx)$-conjugacy classes of $\gl(k,\cx)$-valued polynomials $A(\zeta)=\sum_{i=0}^d A_i\zeta^i$ such that $A(\zeta)$ is regular (i.e. its centraliser is $1$-dimensional) for every $\zeta$ and the characteristic polynomial of $A(\zeta)$ is \eqref{S}.\hfill ${\Box}$\label{Beauville} \end{theorem}

The correspondence is given by associating to a line bundle $E$ on $S$ its direct image $V=\pi_\ast E$, which has a structure of a $\pi_\ast \sO$-module. This is the same as a homomorphism $A:V\rightarrow V(d)$ which satisfies \eqref{S}. The condition $E\in J^{g-1}-\Theta$ is equivalent to $H^0(S,E)=H^1(S,E)=0$ and, hence, to $H^0(\oP^1,V)=H^1(\oP^1,V)=0$, i.e. $V=\bigoplus \sO(-1)$. Thus, we can interpret $A$ as a matricial polynomial precisely when $E\in J^{g-1}-\Theta$.

Somewhat more explicitly, the correspondence is seen from the exact sequence
\begin{equation} 0\rightarrow \sO_\oT(-d)^{\oplus k}\rightarrow \sO_\oT^{\oplus k}\rightarrow E(1)\rightarrow 0, \label{bundle}\end{equation}
where the first map is given by $\eta\cdot 1-A(\zeta)$ and $E(1)$ is viewed as a sheaf on $\oT$ supported on $S$. The inverse map is defined by the commuting diagram
\begin{equation}\begin{CD} H^0\bigl(S,E(1)\bigr) @>>> H^0\bigl(D_{\zeta}, E(1)\bigr)\\ @V \tilde{A}(\zeta) VV  @VV \cdot \eta V \\ H^0\bigl(S,E(1)\bigr) @>>> H^0\bigl(D_{\zeta}, E(1)\bigr), \end{CD}
\label{endom}\end{equation} where $D_{\zeta}$ is the divisor consisting of points of $S$ which lie above $\zeta$ (counting multiplicities).
That the endomorphism $\tilde{A}(\zeta)$ has degree $d$ in $\zeta$ is proved e.g. in \cite{AHH}. 

Given the above proposition, we adopt the following definition:
\begin{definition} A matricial polynomial $A(\zeta)=\sum_{i=0}^d A_i\zeta^i$ is called {\em regular}, if $A(\zeta)$ is a regular matrix for every $\zeta$.\label{regular} \end{definition}
\begin{remark} For a singular curve $S$, Beauville's correspondence most likely extends to $\overline{J^{g-1}}-\overline{\Theta}$, where $\overline{J^{g-1}}$ is the compactified Jacobian  in the sense of \cite{Alex}. It seems to us that this is essentially proved in \cite{AHH}.\label{comp}\end{remark}

Finally, we recall the following fact (see, e.g. \cite{Beau2,Theta}):
\begin{proposition} Let $A(\zeta)$ be the matricial polynomial corresponding to $E\in J^{g-1}-\Theta$. Then $A(\zeta)^T$ corresponds to $E^\ast \otimes K_S$.\hfill $\Box$ \label{canonical}\end{proposition}

\section{Conjugacy classes with respect to subgroups\label{conj}}

We assume that $S\in |\sO(d)|$ is a compact reduced curve of degree $k$. We want to describe the conjugacy classes of matricial polynomials  $A(\zeta)=\sum_{i=0}^d A_i\zeta^i$ with respect to several subgroups of $GL(k,\cx)$: 
\begin{itemize}
\item maximal parabolic:
$$ P=\left\{\begin{pmatrix} g & \ast \\ 0 & m\end{pmatrix};  \enskip g\in GL(k-1,\cx), m\in \cx^\ast\right\};$$
\item maximal reductive:
$$G_{k-1}=\left\{\begin{pmatrix} g & 0 \\ 0 & m\end{pmatrix}; \enskip g\in GL(k-1,\cx), m\in \cx^\ast\right\};$$
\item maximal torus $T$, consisting of diagonal matrices.\end{itemize}

We denote by $U_{g+k-1}$ the open subset of effective (Cartier) divisors $D$ of degree $g+k-1=dk(k-1)/2$, such that $[D](-1)\not\in \Theta$. Observe that Proposition \ref{canonical} implies that, if $L$ is a line bundle of degree $g+k-1$ and $L(-1)\not\in \Theta$, then $L^\ast(dk-d-1)\not \in \Theta$. Thus, if $D\in  U_{g+k-1}$, then the divisor of any section of $\sO_S(dk-d)[-D]$ also lies in $U_{g+k-1}$.
\par
We need the following fact about sections of $\sO_S(l)$, the proof of which proceeds analogously to that \cite[Proposition (4.5)]{Hit} or \cite[Lemma (2.16]{HuMu}.
\begin{lemma} If $l<dk$, then any section $s\in H^0(S,\sO(l))$ may be written uniquely in the form
$$s=\sum_{i=0}^{[l/d]}\eta^i\pi^\ast c_i,$$
where $c_i\in H^0(\oP^1,\sO(l-di))$.\hfill $\Box$\label{O(l)}\end{lemma} 

In particular, any section of $\sO_S(dk-d)$ can be written uniquely as $\sum_{i=0}^{k-1}\eta^i\pi^\ast c_i$, with $c_{k-1}$ being a constant. Denoting by $(s)$ the divisor of a section $s$, we define 
\begin{equation} R_{dk-d}=\left\{\Bigl(\sum_{i=0}^{k-1}\eta^i\pi^\ast c_i\Bigr)\in |\sO_S(dk-d)|;\enskip c_{k-1}\neq 0\right\}.\end{equation}

We can describe conjugacy classes with respect to groups $P$ and $G_{k-1}$ in algebro-geometric terms:

\begin{proposition} Let $S$ be a reduced curve defined by \eqref{S}. Let $M_S$ be the space of regular matricial polynomials $A(\zeta)=\sum_{i=0}^d A_i\zeta^i$, $A_i\in \gl(k,\cx)$, the characteristic polynomial of which is $P(\zeta,\eta)$. There exist natural bijections:
\begin{itemize}
\item[(i)] $M_S/P \simeq U_{g+k-1}$.
\item[(ii)] $M_S/G_{k-1}\simeq \sD=\left\{(D,D^\prime)\in U_{g+k-1}\times U_{g+k-1};\enskip D+D^\prime \in R_{dk-d}\right\}.$
\end{itemize}\label{PG}
\end{proposition}
\begin{remark} The assumption that $S$ is reduced (i.e. without multiple components) is not needed in (ii). Although we make use of it in (i), this is probably also unnecessary.\end{remark} 
\begin{remark} In (i), the natural projection $M_S/P\rightarrow M_S/GL(k,\cx)$ corresponds to the map $U_{g+k-1}\rightarrow J^{g-1}-\Theta$ given by $D\mapsto \sO(dk-d-1)[-D]$.\end{remark}
\begin{remark} In (ii),  $D+D^\prime=\bigl(\det(\eta-A_{(k-1)}(\zeta))\bigr)$, where $A_{(k-1)}$ denotes the upper-left $(k-1)\times (k-1)$-minor of $A$.  Moreover, the projection on either factor realises $\sD$ as a $\cx^{k-1}$-bundle (not a line bundle - it should be viewed as a $(\oP^{k-1}-\oC^{k-2})$-bundle) over $U_{g+k-1}$.\end{remark}

We now discuss $T$-conjugacy classes. We consider a parameterised version of the Gelfand-Zeitlin map \cite{GKL,KW,Bie-Pidst}, and associate to $A(\zeta)$ $k$ spectral curves:
\begin{equation}  S_m=\left\{(\zeta,\eta)\in \oT;\enskip \det\bigl(\eta\cdot 1- A_{(m)}(\zeta)\bigr)=0\right\},\quad m=1,\dots,k, \label{S_m}\end{equation}
where $A_{(m)}(\zeta)$ is the the upper-left $m\times m$ minor of $A(\zeta)$. 

We are going to describe only an open subset of $M_S/T$. Set
\begin{equation} M_S^0=\left\{A(\zeta)\in M_S;\enskip \text{$A_{(m)}(\zeta)$ is regular for every $\zeta$ and every $m=1,\dots,k$}\right\}.\label{MS0}\end{equation}

\begin{proposition} 
Let $S_1,\dots S_{k-1},S_k=S$ be curves in $\oT\simeq \Tot\sO(d)$ with $S_m\in |\sO_\oT(md)|$ for $m=1,\dots,k$. Then there exists a matricial polynomial $A(\zeta)\in M_S^0$ of degree $d$ such that each $S_m$ is defined by \eqref{S_m} if and only if on each $S_m$ with $m\in\{1,\dots,k-1\}$ there exists a divisor $D_m$ of degree $g+m-1=dm(m+1)/2$  satisfying the following conditions 
\begin{itemize} \item[(i)] $D_m$ is a subdivisor of $S_m\cap S_{m+1}$;
\item[(ii)] $H^0\bigl(S_{m}, \sO(dm-d-1)[-D]\bigr)=0$;
\item[(iii)] For each $m\in\{2,\dots,k-1\}$, $[D_{m}-D_{m-1}]\simeq \sO_{S_m}(d)$.\end{itemize}
Moreover, there is a $1-1$ correspondence between these data and $M_S^0/T$.  \label{O(2)}\end{proposition}

The remainder of the section is devoted to a proof of these two propositions. We remark that many arguments are adapted from \cite{HuMu}. 

 We begin with a given matricial polynomial $A(\zeta)$.
\par
Let $F=E(1)$ be the line bundle given by \eqref{bundle}, i.e. $F=\Coker\bigl(\eta-A(\zeta)\bigr)$. Let $G$ be the line bundle corresponding to $A^T(\zeta)$, i.e., owing to Proposition \ref{canonical}, $G=F^\ast\otimes K_S(2)$. Both $F$ and $G$ are line bundles of degree $g+k-1=dk(k-1)/2$, and 
\begin{equation} F\otimes G\simeq K_S(2)\simeq \sO_S(dk-d).\label{FG}\end{equation}

We consider sections $s,s^\prime$ of $F,G$ obtained by projecting the constant section $e_k=(0,\dots,0,1)^T$ of $\sO_\oT^{\oplus k}$. We associate divisors to $A(\zeta)$ via the maps
\begin{equation} \Phi_1:A(\zeta)\mapsto (s^\prime),\quad \Phi_2:A(\zeta)\mapsto \bigl((s),(s^\prime)\bigr).\label{Phi}\end{equation}
From the definition, $\Phi_1$ is $P$-invariant and $\Phi_2$ is $G_{k-1}$-invariant. Moreover, the image of $\Phi_1$ lies $U_{g+k-1}$, while the image of $\Phi_2$ lies in 
$$ \overline\sD=\left\{(D,D^\prime)\in U_{g+k-1}\times U_{g+k-1};\enskip D+D^\prime \in |\sO_S(dk-d)|\right\}.$$
We need to show that $\Phi_2$ maps into $R_{dk-d}$. In fact, we shall show that
\begin{equation} (s)+(s^\prime)=\bigl(\det(\eta-A_{(k-1)}(\zeta))\bigr),\label{D+D}\end{equation}
i.e. $ss^\prime\in H^0(S,\sO(dk-d))$ defines the curve $S_{k-1}$.
\par
Let the subscript ${\rm adj}$ denotes the classical adjoint:
$$ (\eta-A(\zeta))_{\rm adj}(\eta-A(\zeta))=(\eta-A(\zeta))(\eta-A(\zeta))_{\rm adj}=\det(\eta-A(\zeta))\cdot 1.$$
It follows that  $(s)$ coincides with the zero-divisor of  $(\eta-A(\zeta))_{\rm adj}e_k$, the latter being a section of $\Ker(\eta-A(\zeta))\simeq G^\ast$. Let us write 
$$ A=\begin{pmatrix} B &y \\ x & c\end{pmatrix},\quad B\in \gl(k-1,\cx),\enskip  x\in \Hom(\cx^{k-1},\cx), \enskip  y\in \Hom(\cx,\cx^{k-1}),\enskip c\in \cx.$$
 Computing minors along the last row gives
\begin{equation} (\eta-A(\zeta))_{\rm adj}e_k=\begin{pmatrix} \vspace{1mm} -(\eta-B(\zeta))_{\rm adj}y\\  \det(\eta-B(\zeta))\end{pmatrix}.\label{s}\end{equation}
Similarly 
\begin{equation} (\eta-A(\zeta))^T_{\rm adj}e_k=\begin{pmatrix} \vspace{1mm} -(\eta-B(\zeta))^T_{\rm adj}x^T\\  \det(\eta-B(\zeta))\end{pmatrix}.\label{s'}\end{equation}
These formulae imply that $(s)$ and $(s^\prime)$ are subdivisors of $\bigl(\det(\eta-B(\zeta))\bigr)$.

We now use the Weinstein-Aronszajn formula:
\begin{equation} \det(\eta-A(\zeta))=(\eta-c(\zeta))\det(\eta-B(\zeta))-x(\zeta)(\eta-B(\zeta))_{\rm adj}y(\zeta),\label{WA}\end{equation}
from which we conclude $(s)=\bigl((\eta-B(\zeta))_{\rm adj}y(\zeta)\bigr)$, $(s^\prime)=\bigl(x(\zeta)(\eta-B(\zeta))_{\rm adj}\bigr)$. In addition, $\bigl(\det(\eta-B(\zeta))\bigr)=\bigl(x(\zeta)(\eta-B(\zeta))_{\rm adj}y(\zeta)\bigr)$ (as divisors on $S$).  Since at a point of $S\cap S_{k-1}$, $\eta-B(\zeta)$ has corank $1$, $(\eta-B(\zeta))_{\rm adj}$ has rank $1$, and so $(\eta-B(\zeta))_{\rm adj}=uv^T$ for a pair of vectors $u,v$. Therefore $0=x(\zeta)(\eta-B(\zeta))_{\rm adj}y(\zeta)=xuv^Ty$, which means that either $xu=0$ or $v^Ty=0$, and, hence, either $x(\eta-B(\zeta))_{\rm adj}=0$ or  $(\eta-B(\zeta))_{\rm adj}y=0$. This proves \eqref{D+D}.

\medskip

We construct the inverse mapping to $\Phi_1$. We need the following lemma.
\begin{lemma} 
Let $D\in U_{g+k-1}$. There exists a $D^\prime\in U_{g+k-1}$, such that $D+D^\prime\in R_{dk-d}$.\end{lemma}
\begin{proof} (cf. \cite[p.181]{Hit}. Let $s$ be a section defined by $D$. Since $[D](-1)\not \in \Theta$, $s$ does not vanish identically on any fibre of $\pi:S\rightarrow \oP^1$, which consists of $k$ distinct points. Let $\pi^{-1}(\zeta)$ be such a fibre. Since  $L=\sO_S(dk-d-1)[-D]\not\in \Theta$, there exists a section $s^\prime$, which does not vanish at exactly one point $p\in \pi^{-1}(\zeta)$, and we may assume that $s(p)\neq 0$. Thus, $(ss^\prime)(p)\neq 0$ and $ss^\prime$ vanishes at the remaining $k-1$ points of $\pi^{-1}(\zeta)$. On the other hand, if we had  $ss^\prime=\sum_{i=0}^{k-2}\eta^ic_i(\zeta)$, then the vanishing of $ss^\prime$ at $k-1$ points of the fibre would imply that $ss^\prime$ vanishes on the whole fibre, which is a contradiction.\end{proof}

\begin{remark} This lemma is the only place, where the assumption that $S$ is reduced is used.\end{remark}

Let $D\in U_{g+k-1}$. The above lemma and Lemma \ref{O(l)} imply that there exists a section of $\sO_S(dk-d)[-D]$ of the form $\eta^{k-1}+\sum_{i=0}^{k-2}\eta^ic_i(\zeta)$ and, hence, there exists a well-defined $(k-1)$-dimensional subspace $V$ of $H^0(S, \sO(dk-d)[-D])$ of the form $\sum_{i=0}^{k-2}\eta^ic_i(\zeta)$. Computing the endomorphism $A(\zeta)$ with respect to the flag $\{0\}\subset V\subset H^0(S, \sO(dk-d)[-D])$ defines the inverse map $\Phi_1^{-1}:U_{g+k-1}\rightarrow M_S/P$.

\medskip

We can now construct the inverse mapping to $\Phi_2$. Let $(D,D^\prime)\in \sD$. Let $s=\eta^{k-1}+\sum_{i=0}^{k-2}\eta^ic_i(\zeta)$ be the unique, up to a constant multiple, section of $\sO_S(dk-d)$, whose divisor is $D+D^\prime$. 
This time we have a direct sum decomposition $ H^0(S, \sO(dk-d)[-D^\prime])=\cx s\oplus V$, where $V$ is defined as for $\Phi_1^{-1}$. Computing the endomorphism $A(\zeta)$ with respect to this decomposition defines the inverse map $\Phi_2^{-1}:\sD\rightarrow M_S/G_{k-1}$.

\medskip It remains to prove Proposition \ref{O(2)}. The vector space $V$ defined in the construction of $\Phi_2^{-1}$ (for $S=S_k$) can be also viewed as $ H^0(S_{k-1}, \sO(d(k-1))[-D^\prime])$, where $S_{k-1}$ is defined by \eqref{S_m}. Thus, $A_{(k-1)}(\zeta)$ is the endomorphism \eqref{endom} for a particular basis of  $ H^0(S_{k-1}, \sO(d(k-1))[-D^\prime])$. It follows, again from  the construction of $\Phi_2^{-1}$, that $A(\zeta)$ is determined, up to conjugation by the centre of $G_{k-1}$, by $A_{(k-1)}(\zeta)$ and the divisor $D_{k-1}=D^\prime$ on $S_{k-1}$. Applying now the argument to $A_{(k-1)}(\zeta)$, and so on, we get the divisors $D_m$, $m=k-1,\dots,1$, which (together with the curves $S_m$) determine $A(\zeta)$ up to conjugation by $T$. The $D_m$ clearly satisfy conditions (i) and (ii). Moreover, for every $m=1,\dots,k-1$, the matricial polynomial $A_{m}(\zeta)$ corresponds to both $\sO_{S_m}(dm)[-D_m]$ and to $\sO_{S_m}(d(m-1))[-D_{m-1}]$, which proves (iii).

\section{Hyperk\"ahler metrics\label{hk}}

A Riemannian metric is called hyperk\"ahler if its holonomy is a subgroup of $Sp(n)$. Thus, a Riemannian manifold is hyperk\"ahler if it has a triple $I,J,K$ of complex structures, which behave algebraically like a basis of imaginary quaternions, and which are covariant constants for the Levi-Civita connection. We denote by $\omega_I,\omega_J,\omega_K$ the corresponding K\"ahler forms. 
\par
There is a corresponding notion of pseudo-hyperk\"ahler metrics in signature $(4p,4q)$.

\subsection{Twistor space}
A hyperk\"ahler structure on a manifold $M$ can be encoded in an algebraic object - the twistor space $Z$. As a manifold, $Z$ is $M\times S^2$, equipped with a complex structure, which is the standard one on $S^2\simeq \oP^1$, while on the fibre $Z\rightarrow (a,b,c)\in S^2$, it is the complex structure  $aI+bJ+cK$ of $M$. The natural projection $\pi:Z\rightarrow \oP^1$
is holomorphic and $M$ can be identified with a connected component of the space of sections of $\pi$, the normal bundle of which is the direct sum of $\sO(1)$-s and which are invariant under the antipodal map on $S^2$ (which induces an antiholomorphic involution $\sigma$ on $Z$). Such sections are called {\em twistor lines}. Finally, the K\"ahler forms of $M$ combine to define a twisted holomorphic symplectic form on the fibres of $Z$
\begin{equation}\Omega=\left(\omega_J+\sqrt{-1}\omega_K\right)+2\sqrt{-1} \omega_I\zeta+\left(\omega_J-\sqrt{-1}\omega_K \right)\zeta^2,\label{Omega}\end{equation}
where $\zeta$ is the affine coordinate of $\oP^1$. Thus, $\Omega$ is an $\sO(2)$-valued fibrewise symplectic form on $Z$. 

\subsection{K\"ahler potentials}
As remarked above, the complex-valued form $\Omega_I=\omega_J+\sqrt{-1}\omega_K$ is a holomorphic symplectic form for the complex structure $I$. The Darboux theorem holds for such forms and we can find a local $I$-holomorphic chart $u_i,z_i$, $i=1,\dots,n$ such that
\begin{equation} \Omega_I=\omega_J+\sqrt{-1}\omega_K=\sum_{i=1}^n du_i\wedge dz_i.\label{omega}\end{equation}
In this local chart, the K\"ahler form $\omega_I$ can be written as
\begin{equation} \omega_I=\frac{\sqrt{-1}}{2} \sum_{i,j}\left(K_{u_i\bar{u}_j}du_i\wedge d\bar{u}_j+ K_{u_i\bar{z}_j}du_i\wedge d\bar{z}_j+ K_{z_i\bar{u}_j}dz_i\wedge d\bar{u}_j +K_{z_i\bar{z}_j}dz_i\wedge d\bar{z}_j\right), \label{omega1}\end{equation}
for a real-valued function $K$ (we write \eqref{omega1} with positive sign, as in our examples the K\"ahler potential is negative). We see that the complex structure $J$ is given by:
\begin{align} J\left(\frac{\partial}{\partial u_i}\right) =\sum_{j=1}^n\left (K_{z_i\bar{u}_j}\frac{\partial}{\partial \bar{u}_j} + K_{z_i\bar{z}_j}\frac{\partial}{\partial \bar{z}_j}\right)\notag\\
J\left(\frac{\partial}{\partial z_i}\right) =\sum_{j=1}^n\left (-K_{u_i\bar{z}_j}\frac{\partial}{\partial \bar{z}_j} - K_{u_i\bar{u}_j}\frac{\partial}{\partial \bar{u}_j}\right).\label{J} \end{align}
Thus the condition $J^2=-1$ gives a system of nonlinear PDE's for $K$. This system is equivalent to the following condition:
\begin{equation}\begin{pmatrix} K_{u_i\bar{u}_j} & K_{u_i\bar{z}_j}\\ K_{z_i\bar{u}_j} & K_{z_i\bar{z}_j}\end{pmatrix} \in Sp(n,{\Bbb C}),\label{sp}\end{equation}
where the symplectic group is defined with respect to the form \eqref{omega}.
\par
 Conversely, suppose that in some local coordinate system $u_i,z_i$ we have a K\"ahler form $\omega_I$ given by a K\"ahler potential $K$ such that this system of PDE's is satisfied. Then, if we define $\omega_J+i\omega_K$ by the formula \eqref{omega}, we obtain a hyperhermitian structure. However $\omega_J$ and $\omega_K$ are closed, and so, by Lemma 4.1 in \cite{AH}, $J$ and $K=IJ$ are integrable and we have locally a hyperk\"ahler structure. Therefore there is 1-1 correspondence between K\"ahler potentials satisfying the above system of PDE's and local hyperk\"ahler structures.

\subsection{Twistor lines from a K\"ahler potential\label{twistor_lines}}
From the definition of the twistor space, the hyperk\"ahler structure is determined by the twisted form $\Omega$ and by a family of sections of $\pi:Z\rightarrow \oP^1$. Let $\zeta=0$ correspond to the complex structure $I$. We can trivialise the twistor space in a neighborhood of a point in $\pi^{-1}(0)$. Let $\zeta,U_1,\dots,U_n,$$Z_1,\dots,Z_n$ be local holomorphic coordinates, so that $\Omega=\sum_{i=1}^n dU_i\wedge dZ_i$ in these coordinates. A twistor line is now a $2n$-tuple of functions $U_i(\zeta),Z_i(\zeta)$. Since there is a unique twistor line passing through every point of $Z$, the functions $U_i(\zeta),Z_i(\zeta)$ are determined by their values $u_i,z_i$ at $\zeta=0$. It follows now from \eqref{Omega} that the K\"ahler form $\omega_I$, and hence the hyperk\"ahler metric, is determined by the values at $\zeta=0$ of first derivatives of $U_i(\zeta),Z_i(\zeta)$ with respect to $\zeta$. More precisely, if
$$ U_i(\zeta)=u_i+p_i\zeta+\dots,\quad Z_i(\zeta)=z_i+q_i\zeta+\dots,$$
then
$$\omega_I=-\frac{\sqrt{-1}}{2}\sum_{i=1}^n\left(du_i\wedge dq_i+dp_i\wedge dz_i\right).$$
Hence, it is enough to know the twistor lines only up to first order (once we trivialise the twistor space near $\zeta=0$).
\par
If $\omega_I$ is given by a K\"ahler potential $K=K(u_i,\bar u_i,z_i,\bar z_i)$, then, comparing the last formula with \eqref{omega1}, we conclude that, up to additive constants, $p_i=K_{z_i}$ and $q_i=-K_{u_i}$. The freedom of adding an arbitrary constant to each $p_i$ and $q_i$ can be incorporated into the choice of a K\"ahler potential, and so, the twistor lines are given, up to the first order, by:
\begin{equation} U_i(\zeta)=u_i+K_{z_i}\zeta+\dots,\quad Z_i(\zeta)=z_i-K_{u_i}\zeta+\dots.\label{lines}\end{equation}

\section{Generalised Legendre transform\label{glt}}

The generalised Legendre transform, invented by Lindstr\"om and Ro\v{c}ek \cite{LR}, is a construction of (pseudo-)hyperk\"ahler metrics whose twistor space admits a special type of Hamiltonians. 
It generalises the case of $4n$-dimensional hyperk\"ahler manifolds, the symmetry group of which has rank $n$. Recall that 
a tri-Hamiltonian action of a group $H$ on $M$, which extends to a holomorphic action of a complexification $H^\cx$ of $H$ for every complex structure, gives rise to a Hamiltonian $\sigma$-equivariant action of $H^\cx$ on the twistor space $Z$. The moment map is then a section of $ \fH^\cx\otimes \pi^\ast\sO(2)$, where $\fH$ is the Lie algebra of $H$.
\par
It happens however, quite often, that the twisted symplectic form $\Omega$ of a twistor space $Z$, of complex dimension $2n+1$, admits $n$ independent Poisson-commuting  sections $f_i:Z\rightarrow \pi^\ast\sO(2r_i)$, where $r_i$ are no longer constrained to be $1$.  Each $\sO(2r_i)$, $r_i\geq 1$, admits a canonical anti-holomorphic involution $\tau$, induced by that of $\sO(2)\simeq T \oP^1$, and each $f_i$ is assumed to satisfy $\tau\circ f_i=f_i\circ \sigma$ ($\sigma$ is the anti-holomorphic involution on $Z$ induced by the antipodal map of $S^2$). Such ``completely integrable" hyperk\"ahler manifolds $M$ of quaternionic dimension $n$ are produced by the generalised Legendre transform (GLT), which we proceed to describe.
\par
The maps $f_i$ induce maps $\hat{f}_i$ from the space of sections of $Z$, in particular from the manifold $M$,  to the space of sections of $\sO(2r_i)$, i.e. to the space of polynomials of degree $2r_i$, which we write as
$$ \alpha_i(\zeta)=\sum_{a=0}^{2r_i} w_a^i\zeta^a.$$
The real structure $\tau$ acts on this space by \begin{equation}\tau( w_a^i)=(-1)^{r_i+a}\overline{w^i_{2r_i-a}}\label{sigma}\end{equation} and, consequently, we obtain a map
$$ \hat{f}=\bigl(\hat{f}_1,\dots,\hat{f}_n\bigr):M \rightarrow \bigoplus_{i=1}^n {\Bbb R}^{2r_i+1}.$$ 
As explained in \cite{LR, HKLR}, $M$ is a torus (or another abelian group) bundle over the image of $\hat{f}$, where the dimension of the torus is equal to $\#\{i;r_i=1\}$. The image of $\hat{f}$ (and the  hyperk\"ahler structure of $M$) is, in turn,  determined by  a function $F: \bigoplus_{i=1}^n {\Bbb R}^{2r_i+1}\rightarrow {\Bbb R}$ satisfying the
system of PDE's:
\begin{equation} F_{w^i_a,w^j_b}=F_{w^i_c,w^j_d}\label{Feq}\end{equation}
for all $a,b,c,d$ such that $a+b=c+d$. 
\par
An equivalent characterization of \eqref{Feq} is that $F$ is given by a contour integral of a holomorphic (possibly singular or
multivalued) function of $2n+1$ variables $G=G(\zeta,\alpha_1,\dots,\alpha_n)$
\begin{equation}F(w_a^i)=\oint_c G\bigl(\zeta,\alpha_1(\zeta),\dots,\alpha_n(\zeta)\bigr)/\zeta^2 d\zeta\label{Fint} \end{equation}
where $\alpha_i(\zeta)=\sum_{a=0}^{2r_i} w_a^i\zeta^a$, or a sum of such contour integrals.
\par
The function $F$ determines the hyperk\"ahler structure as follows. We have local complex coordinates $z_1,\dots,z_n,u_1,\dots,u_n$ where
\begin{equation} z_i=w_0^i \label{F1}\end{equation}\begin{equation} F_{w_1^i}=\begin{cases} u_i &\text{if  $r_i\geq 2$} \\ u_i+\bar{u}_i &\text{if $r_i=1$}\end{cases} \label{F2}\end{equation}
\begin{equation}F_{w_a^i}= 0\quad \text{if $2\leq a\leq 2r_i-2$}.\label{F3}\end{equation}
\par
Then the K\"ahler potential defined by
\begin{equation}K=F-\sum(u_iw_1^i+\bar{u}_i\bar{w}_1^i)\label{potential} \end{equation}
satisfies the hyperk\"ahler Monge-Amp\`{e}re equations \eqref{sp} and defines the metric of $M$. 
Explicit formulae for the metric in terms of second derivatives of $F$ are given in \cite{LR}. We remark that the subset defined by \eqref{F3} is not necessarily a manifold and, hence, the image of $\hat{f}$ is only an open subset of \eqref{F3}. In addition, one usually needs several functions $F$ defined on overlapping regions of  $\bigoplus_{i=1}^n {\Bbb R}^{2r_i+1}.$

A simple computation shows also that (cf. \eqref{lines})
\begin{equation} \frac{\partial K}{\partial u_i}=-w_1^i,\quad  \frac{\partial K}{\partial z_i}= \frac{\partial F}{\partial w_0^i}.\label{dK}
\end{equation} 

\begin{remark} A given function $F$, satisfying \eqref{Feq}, defines a hyperk\"ahler metric on the set where \eqref{F3} holds {\em and} where the matrix 
\begin{equation}\bigl[F_{w^i_a,w^j_b}\bigr]_{\scriptscriptstyle 0< a < r_i ,\; 0<b<r_j}\label{F''}\end{equation} is invertible (see \cite{LR}). Thus, it is possible that the second condition fails at every point where \eqref{F3} holds, and we do not obtain a hyperk\"ahler metric from $F$.\label{nondeg}\end{remark}

\begin{example} We consider two examples of (non-generalised) Legendre transform in four dimensions. The first, a well known  one, is given by the function $F=2x^2-z\bar{z}$.
It produces the flat metric on $\oR^4$. The second one is the hyperk\"ahler metric obtained from a cubic harmonic polynomial on $\oR^3$:
$$ F=2x^3-3xz\bar{z}.$$
The Legendre transform produces a translation-invariant metric in $I$-holomorphic coordinates $z,u$, with $u+\bar{u}=F_x=6x^2-3z\bar{z}$. The K\"ahler potential for $\omega_I$ is 
$$K=F-xF_x=-4x^3=-\frac{4}{6^{3/2}} (u+\bar{u}+3z\bar{z})^{3/2}.$$
Up to a constant multiple, the metric is
$$ \frac{1}{x}\left(du d\bar u+3zdud\bar z + 3\bar{z} d\bar u dz +(6x^2+9z\bar z)dz d\bar z\right),$$
where $x=\sqrt{(u+\bar u+3z\bar z)/6}$.
It is defined on the subset $\{(u,z);\; u+\bar u+3z\bar z>0\}$ and is non-complete. In addition to the translational symmetry $u\mapsto u+it$, it possesses a circle symmetry $(u,z)\mapsto (u,e^{it}z)$. This hyperk\"ahler metric is 
non-flat. The simplest way to see this is to notice that the surface $z=0$ is a totally geodesic submanifold (since it is the fixed-point set of the circle symmetry). The metric on this surface is $ \sqrt{6}\frac{du d\bar u}{\sqrt{u+\bar u}}$, which is not flat.
\label{harm}\end{example}

\begin{remark} A very natural interpretation of hyperk\"ahler manifolds arising from GLT has been given by Dunajski and Mason \cite{DM2,DM} (see also \cite{GCH}. They show that these manifolds are precisely leaves of the natural hyperk\"ahler foliation of  {\em generalised hyperk\"ahler manifolds}, which have a genuine triholomorphic symmetry.
\end{remark}

\section{GLT on spectral curves\label{glt_spec}} 

As mentioned above, the function $G$ in \eqref{Fint} can be (and usually is) multivalued. Instead of dealing with such functions we can consider single-valued functions on some covering of $\oP^1$. For the time being, we assume that the twistor space $Z$ has a (locally surjective) projection 
\begin{equation} Z\rightarrow \bigoplus_{l=1}^k\bigoplus_{i=1}^{m_l} \sO(2i)\label{all_i},\end{equation}
equivariant with respect to antiholomorphic involutions, so that the induced map on real sections is
\begin{equation}\hat{f}:M\rightarrow \bigoplus_{l=1}^k\bigoplus_{i=1}^{m_l} \oR^{2i+1}.\label{hat_f}\end{equation}
For every $l$, we identify an element of $\bigoplus_{i=1}^{m_l} \oR^{2i+1}$ with a curve $S_l$ in the linear system $|\sO(2m_l)|$, given by the equations
\begin{equation} P_l(\zeta,\eta)=\eta^{m_l}+\sum_{i=1}^{m_l} \alpha_{i}(\zeta)\eta^{m_l-i}=0,\label{curve}\end{equation}
where $\alpha_i(\zeta)=\alpha_i^l(\zeta)$ is a polynomial of degree $2i$  invariant under  \eqref{sigma} and $\eta$ is the fibre coordinate in $\sO(2)\simeq TP^1$. We identify points of $\bigoplus_{l=1}^k\bigoplus_{i=1}^{m_l} \oR^{2i+1}$ with the set $\sS=\sS(m_1,\dots,m_k)$ of $\tau$-invariant singular curves $S=S_1\cup \dots\cup S_k$ satisfying the equation \begin{equation}\prod_{l=1}^k P_l(\zeta,\eta)=0.\label{prod}\end{equation}
\par
We now assume that the function $G$ of \eqref{Fint} can be lifted to  a meromorphic function $G(\zeta,\eta)$ on the singular curve $S=S_1\cup \dots\cup S_k$, and, similarly, the cycle $c$ is also defined on $S$. Thus, the function $F:\sS\rightarrow {\Bbb R}$ is defined by 
\begin{equation}F=\sum_{p=1}^q\oint_{c_p} \frac{1}{\zeta^2}G_p(\zeta,\eta)d\zeta,\label{p}\end{equation} where each $G_p$ is a meromorphic function on $\oT=TP^1$ and $c_p$ is  a homology cycle on the singular curve $S$.  A sufficient condition for $F$ to be real is given by (cf. \cite{HMR}):
\begin{lemma} A function $F$ defined by \eqref{p} is real, provided each $G_p$ and $c_p$ satisfy the following two conditions:
\begin{equation} \oint_{c_p+\tau_\ast c_p} G_p(\zeta,\eta)\frac{d\zeta}{\zeta^2}=0,\end{equation}
\begin{equation}\overline{G_p(\zeta,\eta)}=-\bar{\zeta}^2G_p\left(-\frac{1}{\bar{\zeta}},-\frac{\bar{\eta}}{\bar{\zeta}^2}\right).\label{Greal}\end{equation}
\label{realG}\end{lemma}
\begin{proof} Let us write $L_p$ for the differential in $p$-th term, so that $F=\sum\oint_{c_p}L_p$. Observe that the second condition implies that $\overline{L_p}= -L_p\circ\tau$, which, together with the first condition, shows that $\oint_{c_p}L_p$ is real.\end{proof}
\par
We now discuss constraints \eqref{F3} for an $F$ of the form \eqref{p}.  Recall, that for a smooth curve $S_l$ given by the polynomial \eqref{curve}, a basis of $H^0(S_l,\Omega^1)$ is given by the forms
\begin{equation}\omega_{rs}=\frac{\zeta^r\eta^{s}}{\partial P/\partial\eta}d\zeta\quad 0\leq s\leq m_l-2, \enskip 0\leq r\leq 2(m_l-2)-2s.\label{DFK}\end{equation}
 Let $c$ be a  cycle  on a singular curve $S=S_1\cup\dots\cup S_k$. When restricted to a component $S_l$, $c$ becomes a chain $\gamma_l$ - the sum of a cycle on $S_l$  and oriented paths between intersection points of $S_l$ with other $S_j$.
Let $w^i_a$ be a coefficient of the polynomial $P_l(\zeta,\eta)$. We compute the derivatives $F_{w^i_a}$ at a point $S\in S$, where the component $S_l$ is nonsingular:
$$ \frac{d}{dw_a^i}\oint_c G(\zeta,\eta)\frac{d\zeta }{\zeta^2}=\oint_{c} \frac{1}{\zeta^2}\frac{\partial G}{\partial\eta}\frac{d\eta}{dw_a^i}d\zeta  = \int_{\gamma_l} \frac{1}{\zeta^2}\frac{\partial G}{\partial\eta}\frac{d\eta}{dw_a^i}d\zeta,$$
 and
\begin{equation}\int_{\gamma_l} \frac{1}{\zeta^2}\frac{\partial G}{\partial\eta}\frac{d\eta}{dw_a^i}d\zeta=-\int_{\gamma_l} \frac{\partial G}{\partial\eta}\frac{\zeta^{a-2}\eta^{m_l-i}}{\partial P_l/\partial\eta}d\zeta,\label{der}\end{equation}
where we computed $d\eta/dw_a^i$ by implicit differentiation of \eqref{curve}:
$$\frac{d\eta}{dw_a^i}=-\frac{\zeta^a\eta^{m_l-i}}{\partial P_l/\partial\eta},$$
and, comparing with \eqref{DFK}, we see that for $2\leq a\leq 2i-2$ 
\begin{equation} \frac{d}{dw_a^i}\oint_c G(\zeta,\eta)\frac{d\zeta }{\zeta^2}=-\int_{\gamma_l} \frac{\partial G}{\partial\eta}\omega_{a-2,m_l-i}.\label{d/dw}\end{equation}
Thus, the collection of derivatives $\bigl(\frac{d}{dw_a^i}\oint_c G(\zeta,\eta)\frac{d\zeta }{\zeta^2}\bigr)_{a=2,\dots, 2i-2}$ can be viewed as an element of $H^0(S_l,\Omega^1)^\ast$, and so it defines an element of $\Jac^0(S_l)$. Consequently, the constraints \eqref{F3} for $F$ of the form \eqref{p} imply triviality of certain line bundles on each $S_l$.

\begin{remark} The hyperk\"ahler structures obtained from \eqref{p} have a local $\oR^k$-symmetry.\end{remark}

\subsection{A generalisation\label{gen}}

So far, we have assumed that the twistor space has a locally surjective projection $$ Z\rightarrow \bigoplus_{l=1}^k\bigoplus_{i=1}^{m_l} \sO(2i),$$ so that the induced map on real sections is
\begin{equation*}\hat{f}:M\rightarrow \bigoplus_{l=1}^k\bigoplus_{i=1}^{m_l} \oR^{2i+1},\end{equation*}
and its image has  a non-empty interior. The general form of the GLT, as discussed in \S\ref{glt}, applies also to affine subspaces of $V=\bigoplus_{l=1}^k\bigoplus_{i=1}^{m_l} \oR^{2i+1}$, where we fix the components belonging to some of the $\oR^{2i+1}$. Thus, for any  $R\subset \{(i,l); l=1,\dots,k,\enskip i=1,\dots,m_l\}$, let $p_R$ be the projection of $V$ onto $X_R=\bigoplus_{(i,l)\in R} \oR^{2i+1}$. For any $x\in X_R$, we can apply the GLT to the restriction of a given function $F$ on $V$ to the affine subspace $p_R^{-1}(x)$. Once again, we obtain a pseudo-hyperk\"ahler structure, provided that the matrix \eqref{F''} is invertible.
\par
If $V$ was interpreted as a space of spectral curves, then so is $p_R^{-1}(x)$: some of the coefficients $\alpha_i(\zeta)$ in \eqref{curve} are now fixed. The constraints \eqref{F3} are still equivalent to \eqref{d/dw}, but only with respect to a subspace of the space of holomorphic differentials. 
\par
This generalisation has a simple interpretation in terms of {\em twistor quotients} introduced in \cite{TQ}. It was shown there that that the hyperk\"ahler metrics obtained from the GLT correspond to twistor spaces which admit a compatible fibrewise action of $\sG=\oplus\sO(-2i+2)$. The above passage from GLT on $V$ to GLT on  $p_R^{-1}(x)$ should be viewed as a twistor quotient with respect to a subgroup of $\sG$. The projection $p_R$ is the moment map for this subgroup and $x_R$ is a particular level set for taking the twistor quotient.

\section{Hyperk\"ahler metrics of monopole type \label{hmon}}

We continue to discuss hyperk\"ahler metrics obtained via the generalised Legendre transform from an $F:\bigoplus_{l=1}^k\bigoplus_{i=1}^{m_l} \oR^{2i+1}\rightarrow \oR$. 
As observed in the last section, a rather mild assumption that $F$ arises from a contour integral on a covering of $\oP^1$ defined by \eqref{prod}, implies that the image of $\hat{f}:M\rightarrow \sS$ consists of spectral curves on which certain line bundles are trivial. A well-known example of this situation is the natural metric on the moduli space of charge $n$ $SU(2)$-monopoles, which arises via GLT from the following function \cite{IR, Hough}:
\begin{equation}-\frac{1}{2\pi i}\oint_{\tilde 0} \frac{\eta^2}{\zeta^3}d\zeta + \oint_c \frac{\eta}{\zeta^2}d\zeta,\label{monopole}\end{equation}
where $\tilde 0$ is the sum of simple contours around points in the fibre $\pi_{|S}^{-1}(0)$.
Here $k=1$, i.e. $S=S_1$, $m_1=n$, and the constraints \eqref{F3} correspond (as easily seen from \eqref{d/dw}) to the triviality on $S_1$ of the line bundle $L^2$ with transition function $\exp(2\eta/\zeta)$. Moreover, in this case the twistor lines are determined by sections of $L^2$: essentially the twistor space can be (locally) trivialised near $\zeta=0$, as in section \ref{twistor_lines}, where  $Z_i(\zeta)$ are roots of $P_1(\zeta,\eta)=0$ and $U_i(\zeta)$ is the logarithm of the value of a section of $L^2$ at the point $(\zeta,Z_i(\zeta))\in S_1$. We shall now consider hyperk\"ahler metrics, the twistor lines of which admit a similar description.

We begin by generalising \eqref{monopole}. We want a function $F:\sS=\sS(m_1,\dots,m_k)\rightarrow \oR$. Recall that a point of $\sS$ corresponds to a reducible curve $S=S_1\cup\dots \cup S_k$ given by \eqref{prod}. Let $H_l(\zeta,\eta)$, $l=1,\dots,k$ be a meromorphic function on $\oT$ which is a linear combination of monomials $\eta^i/\zeta^j$, $i,j>0$. Denote by $\tilde 0_l$, $l=1,\dots,k$, the sum of simple contours around points in the fibre $\pi_{|S_l}^{-1}(0)$ of $S_l$, and, finally, let $c$ be a homology cycle on $S$. We define an $F$ on $\sS$ by 
\begin{equation} \oint_c \frac{\eta}{\zeta^2}d\zeta-\frac{1}{2\pi i}\sum_{l=1}^k \oint_{\tilde 0_l} \frac{1}{\zeta^2}H_l(\zeta,\eta)d\zeta.\label{Fbundle}\end{equation} 
For this function, \eqref{d/dw} becomes:
\begin{equation}\frac{dF}{dw_a^i}=\Res_{\tilde 0_l} \frac{\partial H_l}{\partial \eta}\omega_{a-2,m_l-i} - \int_{\gamma_l} \omega_{a-2,m_l-i}, \enskip 2\leq a\leq 2i-2,\label{int-int}\end{equation}
 where $\gamma_l$ is the restriction of $c$ to $S_l$.
The second term in this equation defines an abelian sum, which is identified with $\bigl[\Delta_l^+-\Delta_l^-\bigr]$ in the Jacobian of $S_l$, where $\Delta_l^-$ are the points at which $c$ enters $S_l$ from other curves and  $\Delta_l^+$ the points at which $c$ leaves $S_l$. Thus, via Serre's duality, we have:
\begin{corollary} Let $F:\bigoplus_{l=1}^k\bigoplus_{i=1}^{m_l} \oR^{2i+1}\rightarrow \cx$ be of the form \eqref{Fbundle} and let $S=S_1\cup\dots\cup S_k$ be such that each $S_l$ is nonsingular. Let $E_l$, $l=1,\dots,k$, be the line bundle on $\oT$ with transition function $\exp\frac{\partial H_l}{\partial \eta}$. If the equations \eqref{F3} are satisfied at $S$, then
\begin{equation}{ E_l}_{|S_l}\simeq \left[\Delta_l^+-\Delta_l^-\right].\label{triviality}\end{equation}
on each $S_l$.\hfill $\Box$
\label{answer0}
\end{corollary}
\begin{remark}An $F$ of the form \eqref{Fbundle} will give a hyperk\"ahler metric only if it is real-valued.
It follows from  Lemma \ref{realG} that the first term is real provided $\tau_\ast c=-c$, while the remaining ones are real if each $H_l$ satisfies \eqref{Greal}. This condition is easily seen to be equivalent to each $H_l$ being a sum of 
 $$ \frac{\eta^i}{\zeta^j}+\frac{\eta^i}{\zeta^{2i-j-2}},\quad i>0,\enskip 0<j<2i-2,\enskip \text{$i+j$ is odd}.$$\label{real_sF}\end{remark}
\begin{remark} The expression \eqref{int-int} can be further simplified. Observe namely, that the $(l+1)$-st term in \eqref{Fbundle} is simply $-\Res_{\tilde 0_l} H_l(\zeta,\eta)/\zeta^2$ and so it can be expressed as a symmetric function of the roots $\eta_1,\dots,\eta_{n}$ of $ P_l(\zeta,\eta)$. Hence, this term in \eqref{Fbundle} can be written as a function $\tilde{H}(w^i_a)$ of the coefficients of $P_l$, which makes computing its derivatives trivial. For example, in the case of $F$ given by \eqref{monopole}, one computes easily that the first term is equal to 
$$-\Res_{\tilde 0} \eta^2/\zeta^3=-\Res_0 \frac{\alpha_1(\zeta)^2-2\alpha_2(\zeta)}{\zeta^3},
$$
where $\alpha_1(\zeta), \alpha_2(\zeta)$ are the coefficients in the polynomial $P(\zeta,\eta)=\eta^n +\sum_{i=1}^n \alpha_i(\zeta)\eta^{n-i}$ defining a spectral curve $S$.
Thus, the derivative with respect to $w_a^i$ ($2i-2\geq a\geq 2$)  is $0$ unless $i=2$, and comparing with \eqref{int-int}, we see, as in  \cite{HMR}, that \eqref{F3} are satisfied at $S$ if and only if  
\begin{equation}\int_{c} \omega_{rs}=\begin{cases}2 & \text{if $r=0$ and $s=n-2$}\\ 0 & \text{otherwise}.\end{cases}\label{02}\end{equation}
\label{tildeH}\end{remark}

Observe now that the reality property of each $H_l$ translates into the following reality property of $E_l$:
\begin{equation} E_l\simeq \overline{\tau^\ast E_l},\label{tau^ast}\end{equation}
where ``bar" means taking the opposite complex structure. We can now identify (at least locally) the (hyperk\"ahler) manifold $M$, obtained from \eqref{Fbundle}, with the set of $(S,\nu)$, where $S=S_1\cup\dots\cup S_k$ is a reducible curve such that \eqref{triviality} holds for $l=1,\dots,k$, and $\nu=(\nu_1,\dots,\nu_k)$ with each $\nu_l$ a  section of ${E_l}_{|S_l}\otimes\left[\Delta_l^--\Delta_l^+\right]$ satisfying $\nu_l\overline{\tau^\ast \nu_l}=1\in H^0(S,\sO)$ ($\overline{\tau^\ast \nu_l}$ is a section of ${E_l}^\ast_{|S_l}\otimes\left[\Delta_l^+-\Delta_l^-\right]$). Multiplying each $\nu_l$ by a complex number of modulus one realises $M$ as a $T^k$-bundle over the image of \eqref{hat_f}.
\par
 We claim that the metric on $M$ can be described as for monopole metrics. Let us represent each $\nu_l$ by a pair of functions $f_0^l(\zeta,\eta)$ on $S\cap U_0$ and $f_\infty^l(\tilde \zeta,\tilde \eta)$ on $S\cap U_\infty$  ($U_0=\{\zeta\neq \infty\}$, $U_\infty=\{\zeta\neq 0\}$) satisfying  $f_\infty^l(\tilde \zeta,\tilde \eta)=\exp\frac{\partial H_l}{\partial \eta}f_0^l(\zeta,\eta)$ over $\zeta\neq 0,\infty$. Let $\bigl(\zeta,\eta^l_j(\zeta)\bigr)$, $j=1,\dots,m_l$, be the points of $S_l$ over $\zeta$. For $\zeta$ near $0$ define a $\zeta$-dependent complex-symplectic form by:
\begin{equation}\Omega(\zeta)=\sum_{l=1}^k\sum_{j=1}^{m_l}\frac{df_0^l\bigl(\zeta,\eta_j^l(\zeta)\bigr)}{f_0^l\bigl(\zeta,\eta_j^l(\zeta)\bigr)}\wedge d\eta_j^l(\zeta).\label{Omega2}\end{equation}
We claim 
\begin{theorem} On the subset of $M$, where the fibre of $S$ over $\zeta=0$ consists of distinct points and each $S_l$ is smooth, local $I$-complex coordinates ($I$ is the complex structure corresponding to $\zeta=0$) are given by 
$\eta^l_j(0),\log f_0^l\bigl(0,\eta^l_j(0)\bigr)$, $l=1,\dots,k$, $j=1,\dots,m_l$. Moreover, the complex-symplectic form $\omega_J+i\omega_K$ is equal to $\Omega(0)$, while the K\"ahler form $\omega_I$ is equal to $\frac{1}{2\sqrt{-1}\frac{d\Omega(\zeta)}{d\zeta}}_{|\zeta=0}$.
\label{one}\end{theorem}

We shall prove this theorem together with its inverse, which we now state. 
Let $E_1,\dots,E_k$ be line bundles on $\oT$ satisfying the reality property \eqref{tau^ast}. Let $V$ be an open subset of $\sS=\sS(m_1,\dots,m_k)$ such that all $S\in V$ are isotopic, $S_l$ are smooth, and no intersection points lie over $\zeta=0$ or $\zeta=\infty$.
Suppose that to any $S\in V$ and any $l=1,\dots,k$ we associated disjoint divisors $\Delta_l^-$ and $\Delta_l^-=\tau(\Delta_l^+)$ which are subdivisors of the divisor cut out on $S_l$ by other $S_i$.
\begin{theorem} Suppose that $M$ is a hyperk\"ahler manifold, which as a manifold is the set of $(S,\nu)$, where $S=S_1\cup\dots\cup S_k \in V$ satisfy the above assumptions and \eqref{triviality} holds for $l=1,\dots,k$, and $\nu=(\nu_1,\dots,\nu_k)$ with each $\nu_l$ a  section of ${E_l}_{|S_l}\otimes\left[\Delta_l^--\Delta_l^+\right]$ satisfying $\nu_l\overline{\tau^\ast \nu_l}=1\in H^0(S,\sO)$. Suppose also that the twistor space of $M$ is trivialised near $\zeta=0$ so that the twisted symplectic form is given by \eqref{Omega2}.
\par
Then there exists a homology cycle $c$ on $S$, with $\tau_\ast c=-c$, entering each $S_l$ at points of $\Delta_l^-$ and leaving it at points of $\Delta_l^+$, and such that the hyperk\"ahler metric of $M$ is produced by the generalised Legendre transform applied to the function \eqref{Fbundle}.
\label{two}\end{theorem}

The remainder of the section is devoted to a proof of these two theorems. The basic idea comes from \cite{IR,CK,Hough}.
\par
We begin by discussing the situation on a single component $S_l$. Thus, we assume that we consider a smooth curve $C$ given by the equation
\begin{equation*} \eta^m+\sum_{j=1}^m \alpha_j(\zeta)\eta^{m-j}=0,\quad \alpha_j(\zeta)=\sum_{a=0}^{2j}w_a^j\zeta^a.\end{equation*}
 We assume that we have a meromorphic section $\nu$ of a line bundle on $C$, which is the restriction of a line bundle $E$ on $\oT$, with the transition function $\exp \frac{\partial H}{\partial \eta}$. We represent $\nu$ by a pair of meromorphic functions $f_0(\zeta,\eta)$ on $S\cap U_0$ and $f_\infty(\tilde \zeta,\tilde \eta)$ on $S\cap U_\infty$  satisfying  $f_\infty(\tilde \zeta,\tilde \eta)=\exp\frac{\partial H}{\partial \eta}f_0(\zeta,\eta)$ over $\zeta\neq 0,\infty$. Let  $ \Delta^+$ be the zero divisor of $\nu$ and $ \Delta^-$ its polar divisor. We assume that they are disjoint from the fibres of $C$ over $\zeta=0$ and $\zeta=\infty$. For the time being, we do not require any reality conditions.
\par
We consider the form \eqref{Omega2}, which we rewrite in terms of coefficients of the polynomial rather than its roots (cf. \cite{Hough}):
\begin{equation} \sum_{j=1}^{m}\frac{df_0\bigl(\zeta,\eta_j(\zeta)\bigr)}{f_0\bigl(\zeta,\eta_j(\zeta)\bigr)}\wedge d\eta_j(\zeta)= \sum_{j=1}^{m} dU_j(\zeta)\wedge d\alpha_j(\zeta),\label{Omega3}\end{equation}
where
\begin{equation} U_j(\zeta)=\sum_{i=1}^m \frac{\partial \eta_i(\zeta)}{\partial \alpha_j}\log f_0\bigl(\zeta,\eta_i(\zeta)\bigr).\label{U}\end{equation}
Note that $U_j$ is a (multi-valued) function on $\oP^1$.
We now cut the surface $C$ as in \cite[pp. 242--243]{GH}, so that $\log f_0$ and the $U_j$ become  single valued functions. Let $a_1,\dots, a_g$,$b_1,\dots,b_g$ be cycles on $C$ representing the canonical basis of $H_1(C,\oZ)$,  disjoint except the common base point $s_0\in C$ and not containing any zero or pole of $\nu$. Let $\epsilon_i$ be smooth arcs from $s_0$ to the points $\{p_i\}$ in the support of $(\nu)$, disjoint from all $a_r,b_r$ (except for $s_0$). 
We may also assume that that $a_r,b_r,\epsilon_i$  do not contain any points of the fibres over $0,\infty$. Then the complement $P$ of all these paths is a simply connected region as drawn below  (cf. \cite[p. 242]{GH}).

\medskip

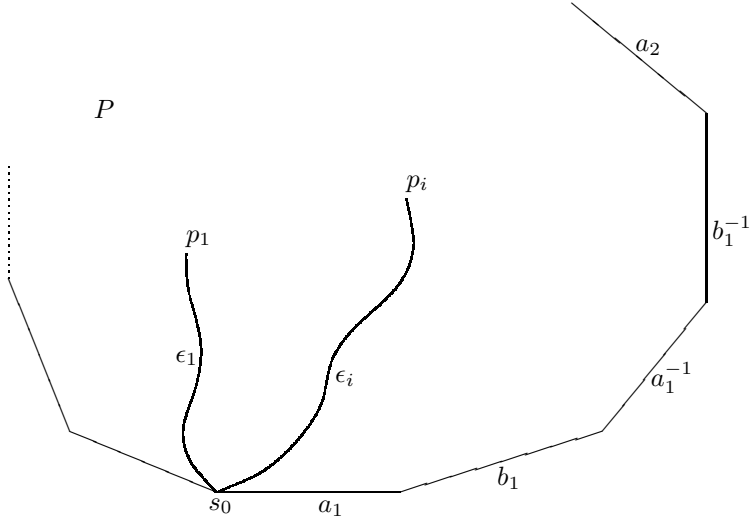
\setlength{\unitlength}{0.00043333in}
\begingroup\makeatletter\ifx\SetFigFont\undefined%
\gdef\SetFigFont#1#2#3#4#5{%
  \reset@font\fontsize{#1}{#2pt}%
  \fontfamily{#3}\fontseries{#4}\fontshape{#5}%
  \selectfont}%
\fi\endgroup%
{\renewcommand{\dashlinestretch}{30}
\begin{picture}(8574,6039)(-500,-10)
\put(2462,-200){\shortstack{$s_0$}}
\put(3802,-230){\shortstack{$a_1$}}
\put(6000,102){\shortstack{$b_1$}}
\put(7882,1320){\shortstack{$a_1^{-1}$}}
\put(8632,3120){\shortstack{$b_1^{-1}$}}
\put(7692,5420){\shortstack{$a_2$}}
\put(2187,3087){\shortstack{$p_1$}}
\put(4887,3762){\shortstack{$p_i$}}
\put(4012,1362){\shortstack{$\epsilon_i$}}
\put(2056,1617){\shortstack{$\epsilon_1$}}
\put(1056,4600){\shortstack{$\Large{P}$}}

\dottedline{80}(12,4000)(12,2637)
\drawline(12,2637)(762,762)(2562,12)
	(4812,12)(7287,762)(8562,2337)
	(8562,4662)(6912,6012)
\drawline(2562,12)(2560,14)(2555,19)
	(2547,27)(2535,40)(2519,57)
	(2500,77)(2478,101)(2455,126)
	(2431,152)(2407,178)(2385,203)
	(2363,227)(2343,250)(2325,272)
	(2309,292)(2294,312)(2281,330)
	(2268,348)(2257,365)(2247,382)
	(2237,399)(2227,419)(2218,438)
	(2209,458)(2200,478)(2193,500)
	(2185,521)(2179,544)(2173,567)
	(2168,590)(2163,614)(2160,637)
	(2157,661)(2155,685)(2154,708)
	(2153,730)(2154,753)(2155,774)
	(2156,795)(2159,816)(2162,837)
	(2165,856)(2169,875)(2174,894)
	(2179,914)(2185,934)(2191,955)
	(2197,976)(2205,998)(2212,1020)
	(2220,1043)(2228,1065)(2236,1088)
	(2245,1112)(2253,1135)(2261,1157)
	(2269,1180)(2277,1203)(2285,1225)
	(2292,1247)(2299,1269)(2306,1290)
	(2312,1312)(2318,1332)(2323,1353)
	(2328,1374)(2333,1396)(2338,1418)
	(2343,1441)(2347,1464)(2351,1489)
	(2355,1514)(2358,1539)(2361,1565)
	(2364,1590)(2366,1617)(2368,1643)
	(2370,1669)(2371,1694)(2371,1720)
	(2371,1745)(2371,1770)(2370,1794)
	(2369,1818)(2367,1841)(2365,1864)
	(2362,1887)(2359,1910)(2355,1933)
	(2351,1956)(2347,1980)(2342,2004)
	(2337,2029)(2331,2054)(2325,2080)
	(2318,2105)(2312,2131)(2305,2157)
	(2298,2184)(2291,2209)(2284,2235)
	(2277,2260)(2270,2285)(2263,2310)
	(2257,2333)(2250,2356)(2245,2378)
	(2239,2400)(2234,2421)(2229,2442)
	(2225,2462)(2220,2486)(2215,2510)
	(2211,2534)(2208,2558)(2205,2584)
	(2202,2611)(2199,2640)(2197,2670)
	(2195,2703)(2193,2737)(2192,2771)
	(2191,2806)(2189,2839)(2189,2869)
	(2188,2894)(2187,2913)(2187,2926)
	(2187,2934)(2187,2937)
\drawline(2562,12)(2565,13)(2570,15)
	(2581,19)(2597,26)(2618,34)
	(2645,45)(2676,57)(2710,71)
	(2746,86)(2783,101)(2819,117)
	(2855,132)(2889,146)(2921,161)
	(2951,174)(2979,187)(3005,200)
	(3030,213)(3053,225)(3075,237)
	(3096,249)(3117,262)(3137,274)
	(3157,288)(3177,301)(3197,315)
	(3217,330)(3237,345)(3257,361)
	(3278,377)(3298,394)(3319,412)
	(3339,430)(3360,448)(3380,467)
	(3400,486)(3420,505)(3439,524)
	(3458,543)(3476,562)(3493,581)
	(3510,599)(3527,617)(3542,635)
	(3558,652)(3573,670)(3587,687)
	(3603,706)(3618,725)(3633,744)
	(3648,764)(3664,784)(3679,804)
	(3694,825)(3709,846)(3723,868)
	(3738,890)(3752,912)(3765,934)
	(3778,956)(3790,978)(3802,999)
	(3813,1020)(3823,1040)(3832,1060)
	(3840,1080)(3848,1099)(3855,1118)
	(3862,1137)(3869,1158)(3875,1179)
	(3880,1200)(3885,1222)(3890,1244)
	(3895,1267)(3899,1290)(3904,1314)
	(3908,1338)(3912,1362)(3916,1386)
	(3920,1410)(3925,1434)(3929,1457)
	(3934,1480)(3939,1502)(3944,1524)
	(3949,1545)(3955,1566)(3962,1587)
	(3969,1606)(3976,1625)(3984,1644)
	(3992,1664)(4002,1684)(4012,1704)
	(4023,1725)(4035,1746)(4048,1768)
	(4062,1790)(4076,1812)(4091,1834)
	(4106,1856)(4122,1878)(4139,1899)
	(4156,1920)(4173,1940)(4190,1960)
	(4208,1980)(4225,1999)(4243,2018)
	(4262,2037)(4279,2054)(4297,2072)
	(4316,2089)(4335,2107)(4355,2125)
	(4376,2144)(4397,2163)(4419,2183)
	(4442,2203)(4464,2223)(4487,2243)
	(4510,2264)(4533,2284)(4556,2305)
	(4578,2325)(4600,2346)(4621,2366)
	(4642,2386)(4662,2405)(4681,2425)
	(4699,2444)(4717,2462)(4733,2481)
	(4750,2499)(4766,2520)(4783,2541)
	(4798,2562)(4814,2584)(4829,2606)
	(4843,2629)(4857,2653)(4871,2677)
	(4884,2702)(4896,2727)(4907,2753)
	(4918,2779)(4927,2804)(4936,2830)
	(4944,2856)(4951,2881)(4957,2906)
	(4962,2930)(4967,2954)(4970,2978)
	(4973,3001)(4975,3024)(4976,3046)
	(4976,3068)(4976,3090)(4975,3113)
	(4974,3138)(4971,3163)(4968,3190)
	(4964,3219)(4960,3250)(4954,3282)
	(4948,3317)(4942,3353)(4934,3390)
	(4927,3427)(4919,3463)(4912,3497)
	(4906,3529)(4900,3555)(4895,3577)
	(4891,3593)(4889,3603)(4888,3609)(4887,3612)
\end{picture}

\bigskip

\bigskip

 We view $P$ as a polygon with sides $a_r,a_r^{-1},b_r,b_r^{-1},\epsilon_i,\epsilon_i^{-1}$. We choose a single-valued branch of $\log f_0$ on $P$. 
The  function $\kappa_j(\zeta,\eta)=\frac{\partial \eta}{\partial \alpha^j}\log f_0(\zeta,\eta) $ is a meromorphic function on $P$ and we have 
$U_j(\zeta)=\sum_{i=1}^m\kappa_j(\zeta,\eta_i)$. 
Now, for an integer $s$:
$$\oint_0 U_j\frac{d\zeta}{\zeta^{s}}=\oint_{\tilde{0}}\kappa_j\frac{d\zeta}{\zeta^{s}},$$
where $0$ is a simple cycle around $0\in \oP^1$ and $\tilde 0$ its lift to $C$.
Since the differential on the right-hand side has poles only at $\zeta=0$ and $\zeta=\infty$, we have
$$\oint_{\tilde{0}}\kappa_j\frac{d\zeta}{\zeta^{s}}=\oint_{\partial P}\kappa_j\frac{d\zeta}{\zeta^{s}}- \oint_{\tilde{\infty}}\kappa_j\frac{d\zeta}{\zeta^{s}}.$$
On the other hand, the patching formula for $\kappa_j$ is
$$\tilde\kappa_j=\zeta^{2j-2}\left(\kappa_j+\frac{\partial H}{\partial \alpha_j}\right),$$
and hence
$$\oint_{\tilde{0}}\kappa_j\frac{d\zeta}{\zeta^{s}}=\oint_{\partial P}\kappa_j\frac{d\zeta}{\zeta^{s}}- \oint_{\tilde{\infty}}\tilde\kappa_j\frac{d\tilde\zeta}{\zeta^{s+2j-4}} + \oint_{\tilde{\infty}}\frac{\partial H}{\partial \alpha_j}\frac{d\zeta}{\zeta^{s}}.$$

The integrand in the third term arises, as observed in Remark \ref{tildeH}, from a function on $\oP^1$ and, so it can be replaced thanks to the residue theorem, by the integral around $0$. Thus:
\begin{equation} \oint_0 U_j(\zeta)\frac{d\zeta}{\zeta^{s}} + \oint_\infty U_j(\tilde \zeta){\tilde\zeta}^{s+2j-4}d{\tilde\zeta}=\oint_{\partial P}\kappa_j\frac{d\zeta}{\zeta^{s}} - \oint_{\tilde{0}}\frac{\partial H}{\partial \alpha_j}\frac{d\zeta}{\zeta^{s}}.\label{first}\end{equation}

We now compute $ \oint_{\partial P}\kappa_j\frac{d\zeta}{\zeta^{s}}$. We rewrite it as $\oint_{\partial P} \log f_0\psi$, where $\psi$  is a meromorphic differential (equal to $\frac{\partial \eta}{\partial \alpha^j}\frac{d\zeta}{\zeta^{s}}$), and compute it as in \cite[p. 243]{GH}:
for points $p\in a_r$, $p^\prime\in a_r^{-1}$, identified on $C$, we have
$$\log f_0(p^\prime)=\log f_0(p)+\int_{b_{r}} d\log f_0,$$
and so
$$\int_{a_r+a_r^{-1}}\log f_0\psi=-\oint_{b_r}d\log f_0\cdot \oint_{a_r}\psi.$$
Similarly, 
$$\int_{b_{r}+b_{r}^{-1}}\log f_0\psi=\oint_{a_r}d\log f_0 \cdot\oint_{b_r}\psi.$$

For points $p\in \epsilon_i$, $p^\prime\in \epsilon_i^{-1}$ identified on $C$,
$$ \log f_0(p^\prime)-\log f_0(p)=-2\pi \sqrt{-1}\ord_{p_i}(f_0),$$
and hence
$$\int_{\epsilon_i+\epsilon_i^{-1}}\log f_0\psi=2\pi \sqrt{-1}\ord_{p_i}(f_0)\int_{s_0}^{p_i}\psi.$$
Since the integral of $\log f_0$ over a homology cycle is an integer multiple of $2\pi \sqrt{-1}$, we have a well defined homology cycle in $H_1(C,\oZ)$, represented by
\begin{equation}\lambda=\frac{1}{2\pi \sqrt{-1}}\left( -\sum_r\left(\oint_{b_r}d\log f_0\right)a_r+\sum_r\left(\oint_{a_r}d\log f_0\right)b_r\right).\label{lambda}\end{equation}
If we define a chain $\gamma$ as 
\begin{equation}\gamma=\lambda+\sum_i \ord_{p_i}(f_0)\epsilon _i,\label{gamma}\end{equation}
 then we obtain, from the above calculations,
\begin{equation*} \frac{1}{2\pi \sqrt{-1}}\oint_0 U_j\frac{d\zeta}{\zeta^{s}}+  \frac{1}{2\pi \sqrt{-1}}\oint_\infty U_j(\tilde \zeta){\tilde\zeta}^{s+2j-4}d{\tilde\zeta}=\int_\gamma \frac{\partial \eta}{\partial \alpha_j}\frac{d\zeta}{\zeta^{s}}- \frac{1}{2\pi \sqrt{-1}}\oint_{\tilde{0}}\frac{\partial H}{\partial \alpha_j}\frac{d\zeta}{\zeta^{s}}.\label{second}\end{equation*}
We now define a function $\phi$ on a neighbourhood of $(C,\Delta^+,\Delta^-)$ in $$\{(S,D^+,D^-);\: S\in|\sO(2m)|,\; \text{$D^\pm$ - divisors on $S$ of the same degree as $\Delta^\pm$}\}$$
by
$$\phi(S,D^+,D^-)=\int_\gamma\eta \frac{d\zeta}{\zeta^{2}}-\frac{1}{2\pi \sqrt{-1}}\oint_{\tilde{0}}H(\zeta,\eta)\frac{d\zeta}{\zeta^{2}}.$$
Since $\frac{\partial \alpha_j}{\partial w_a^j}=\zeta^a$, we get at $(S,D^+,D^-)=(C,\Delta^+,\Delta^-)$, (setting $s=2-a$):
\begin{equation*} \frac{1}{2\pi \sqrt{-1}}\oint_0 U_j(\zeta)\zeta^{a-2}d\zeta +  \frac{1}{2\pi \sqrt{-1}}\oint_\infty U_j(\tilde \zeta){\tilde\zeta}^{2j-2-a}d{\tilde\zeta}=\frac{\partial \phi}{\partial w_a^j}-R(\Delta^+,\Delta^-),\label{third}\end{equation*}
where 
$$R(\Delta^+,\Delta^-)=\sum_{(\zeta,\eta)\in \Delta^+}\frac{\eta}{\zeta^2}\frac{\partial \zeta}{\partial w_a^j}-\sum_{(\zeta,\eta)\in \Delta^-}\frac{\eta}{\zeta^2}\frac{\partial \zeta}{\partial w_a^j}.$$
 Hence
\begin{equation} \frac{\partial \phi}{\partial w_a^j}-R(\Delta^+,\Delta^-)=\begin{cases} {\frac{d U_j(\zeta)}{d\zeta}}_{|\zeta=0}& \text{if $a=0$}\\ U_1(0)+ U_1(\infty)& \text{if $a=1$ and $j=1$}\\U_j(0) & \text{if $a=1$ and $j>1$}\\  0 & \text{if $2\leq a\leq 2j-2$}.\end{cases}\label{crux}\end{equation}
\par
We now impose reality conditions: we assume that the curve $C$ is $\tau$-invariant, the line bundle $E$ satisfies \eqref{tau^ast} and that  $\nu\overline{\tau^\ast \nu}=1$. In particular, $\tau(\Delta^-)=\Delta^+$.
\par
First of all, we can choose a canonical basis of $H_1(C,\oZ)$ for which
\begin{equation} \tau_\ast(a_r)=-a_r,\enskip\tau_\ast(b_r)=b_r,\quad r=1,\dots,g.\label{ab_real}\end{equation}
This follows (cf. \cite[p.227]{HMR}) from two facts: (1) since $\tau$ is anti-holomorphic, the intersection number of any two cycles satisfies $\#(\lambda,\mu)=-\#(\tau_\ast\lambda,\tau_\ast\mu)$, and (2) $\tau_\ast$ is diagonalisable. Thus, we take $a_r$ to be the  $(-1)$-eigenvectors and $b_r$ to be $1$-eigenvectors of $\tau_\ast$. 
\par
We now have $d\log f_0=-\overline{\tau^\ast d\log f_0}$, and, hence, 
$$ \oint_{a_r}d\log f_0=-\oint_{a_r}\overline{\tau^\ast d\log f_0}=-\oint_{\tau_\ast(a_r)}\overline{ d\log f_0}=\overline{\oint_{a_r}d\log f_0}.$$
Therefore $ \oint_{a_r}d\log f_0$ is real, but, since it is also an integer multiple of $2\pi \sqrt{-1}$, it must be equal to zero. Hence, the cycle \eqref{lambda} is, in this canonical basis, a linear combination of the $a_r$ only, and so $\tau_\ast\lambda=-\lambda$. Moreover, as  $\tau(\Delta^-)=\Delta^+$, we can replace $\sum_i \ord_{p_i}(f_0)\epsilon _i$ in \eqref{gamma} by paths going from $p_i$ to $\tau(p_i)$, so that $\tau_\ast(\gamma)=-\gamma$. 

\medskip

We now prove Theorem \ref{one}. We have the function $F$ given by \eqref{Fbundle}, and we know, from Corollary \ref{triviality} and the reality conditions, that on each $S_l$ there is a section $\nu_l$ of ${E_l}_{|S_l}\otimes\left[\Delta_l^--\Delta_l^+\right]$ with $\nu_l\overline{\tau^\ast \nu_l}=1$. From the above calculations, applied to every component $S_l$, we obtain another function $F^\prime= \sum_{l=1}^k \phi({S_l})$, which, apriori, may differ from $F$ in the choice of the cycle. Let $c^\prime$ be the cycle for $F^\prime$, i.e. $c^\prime$ is the sum of $\gamma$'s on different $S_l$. We denote the restriction of  $c^\prime$ to $S_l$ by $\gamma^\prime_l$. Computing the second term in \eqref{Fbundle}, as in Remark \ref{tildeH}, we conclude, from \eqref{F3} and \eqref{int-int} that
$$\int_{\gamma_l-\gamma^\prime_l}\Omega=0$$
for every $l$ and every holomorphic differential $\Omega$ on $S_l$. The paths components on each $S_l$ are determined by the singularities of $\nu_l$, and hence, they are the same for  ${\gamma_l}$ and for ${\gamma^\prime_l}$. Thus 
$$\int_{\lambda_l-\lambda^\prime_l}\Omega=0,$$
where $\lambda_l,\lambda^\prime_l$ are the contour components of  ${\gamma_l},{\gamma^\prime_l}$. From the above discussion, with our choice of the basis of $H_1(S_l,\oZ)$, both $\lambda_l$ and $\lambda^\prime_l$ are combination of the $a_r$ only. Therefore $\lambda_l=\lambda^\prime_l$ on every $S_l$, and, consequently, $F=F^\prime$. Theorem \ref{one} follows from \eqref{crux}, \eqref{dK}, and \eqref{lines}.

\medskip

To prove Theorem \ref{two}, we define the function $F$ in \eqref{Fbundle} as $ \sum_{l=1}^k \phi({S_l})$ , and the cycle $c$ on $S$ is the sum of $\gamma$'s on different $S_l$.  Theorem \ref{two} follows easily from \eqref{crux}, \eqref{dK}, and \eqref{lines} (all terms of the form $ R(\Delta^+,\Delta^-)$, arising from different $S_l$,  cancel).

\section{Examples}

\subsection{$SU(2)$-monopole metrics and asymptotic monopole metrics}

The moduli space of $SU(2)$-monopoles of charge $n$ is a $4n$-dimensional complete hyperk\"ahler manifold, biholomorphic to the space of rational maps $\oP^1\rightarrow \oP^1$ of degree $n$. When we vary the complex structure, the denominator of the rational map, corresponding to a given monopole, traces a curve $S\in |\sO(2n)|$. Hitchin \cite{Hit} shows that the line bundle $L^2$ on $T$ with transition function $\exp(2\eta/\zeta)$ is trivial on $S$. The monopole metric is a basic example of  Theorem \ref{two}.   Indeed, it has been shown by Ivanov and Ro\v{c}ek \cite{IR} (for $n=2$) and by Houghton \cite{Hough} (for arbitrary $n$) that the monopole metric can be constructed via the generalised Legendre transform from the function 
\begin{equation} F=-\frac{1}{2\pi i}\oint_{\tilde 0} \frac{\eta^2}{\zeta^3}d\zeta + \oint_c \frac{\eta}{\zeta^2}d\zeta, \label{Fmon}\end{equation}
on $\sS(n)$.

\medskip

It has been known since the work of Taubes \cite{Tau} that the infinity of the moduli space of (centred) monopoles corresponds to a monopole decaying to a superposition of monopoles of lower charges. Thus, for any partition $(n_1,\dots,n_k)$ of $n$, there is an asymptotic region of the monopole moduli space, where monopoles are approximately a superposition of $k$ monopoles of charges $n_1,\dots,n_k$. To understand the asymptotic dynamics, we make a guess that the metric approximates the metric given by \eqref{Fmon}, but this time defined on unions of spectral curves of degrees $n_1,\dots,n_k$. In other words, this time $F$ is defined on $\sS(n_1,\dots,n_k)$. Corollary \ref{answer0} implies that the condition \eqref{F3} is equivalent to $L^2_{|S_{l}}\simeq [\Delta_l^+-\Delta_l^-]$ for every $l$, where $\Delta_l^++\Delta^-_l$ is the divisor cut out on $S_l$ by the other curves. These are, indeed, the constraints for the asymptotic monopole metrics considered in \cite{clusters} and Theorem \ref{two} shows that the metrics produced by the GLT in this case are  those in \cite{clusters}. In fact, it was this GLT approach which first suggested what the asymptotic monopole metrics should be.

\medskip
  
We observe that Remark \ref{tildeH} applies to these asymptotic metrics as well, and the constraints \eqref{F3} are equivalent to \eqref{02} being valid on every $S_l$, $l=1,\dots,k$.

\subsection{$SU(N)$-monopole metrics} We recall the twistor description, due to Hurtubise and Murray, of the moduli space of $SU(N)$-monopoles with maximal symmetry breaking. An $SU(N)$-monopole has a magnetic charge $(m_1,\dots,m_{N-1})$ and its Higgs field at infinity is conjugate to $\sqrt{-1}\diag(\mu_1,\dots,\mu_N)$, with $\mu_1<\mu_2<\dots<\mu_N$. A generic monopole with these data  corresponds to a collection of $\tau$-invariant compact spectral curves $S_p\in |\sO(2m_p)|$, $p=1,\dots,N-1$, in generic position, along with a splitting $S_p\cap S_{p-1}=S_{p,p-1}\cup S_{p-1,p}$ into subsets of disjoint cardinality, such that $\tau(S_{p,p-1})= S_{p-1,p}$ and, over $S_p$,
$$ L^{\mu_{p+1}-\mu_p}(m_{p-1}+m_{p+1})[-S_{p,p+1}-S_{p-1,p}]\simeq \sO,$$
where $L^s$ is the line bundle defined in the previous subsection. In addition, there are vanishing and positivity conditions - see \cite[p.38]{HuMu}.
\par
The moduli space of $SU(N)$-monopoles with fixed $(m_1,\dots,m_{N-1})$ and $(\mu_1,\dots,\mu_N)$ has a natural hyperk\"ahler metric. It follows from the work of Hurtubise and Murray \cite{HuMu} that the Nahm transform induces a biholomorphism between the twistor space of this metric and the twistor space of the natural $L^2$ metric on the moduli space of solutions to Nahm's equations. Moreover, this biholomorphism commutes with the the real stucture, preserves the twistor lines and the fibres of the projections onto $\oP^1$. Therefore the Nahm transform preserves the hypercomplex structure and, whence, the Levi-Civita connection. \par
It follows now, by comparing  \cite[\S 3]{HuMu} with \cite{BielCMP}, that the metric on the moduli space of solutions to Nahm's equations can be described in terms the above spectral data, by the formula  \eqref{Omega2}, where $\bigl(\zeta,\eta^l_j(\zeta)\bigr)$, $j=1,\dots,m_l$, are the points of $S_l$ over $\zeta$, $l=1,\dots,N_1$, and $f_0^l(\zeta,\eta)$ represents a section $\sigma_l$ of 
\begin{equation} L^{\mu_{l+1}-\mu_l}(m_{l+1}-m_{l-1})[-S_{l,l+1}+S_{l-1,l}],\label{L^mu}\end{equation}
satisfying $\sigma_l\overline{\tau^\ast \sigma_l}=\frac{P_{l+1}}{P_{l-1}}$, where $P_l=P_l(\zeta,\eta)$ is the polynomial defining $S_l$. Let $\nu_l=\sigma_l/\overline{\tau^\ast \sigma_l}$, $l=1,\dots, N-1$. It satisfies $\nu_l\overline{\tau^\ast \nu_l}=1$ and it is a section of 
$$ L^{\mu_{l+1}-\mu_l}(m_{l+1}-m_{l-1})[-S_{l,l+1}+S_{l-1,l}]\otimes \left(L^{-\mu_{l+1}+\mu_l}(m_{l+1}-m_{l-1})[-S_{l+1,l}+S_{l,l-1}]\right)^\ast,$$ 
i.e. of \begin{equation}L^{2\mu_{l+1}-2\mu_l}[S_{l+1,l}+S_{l-1,l}-S_{l,l+1}-S_{l,l-1}].\label{L^2mu}\end{equation}
Moreover $\nu_l$ is represented, on $\{\zeta\neq 0\}$, by $\tilde{f}_0^l(\zeta,\eta)=\bigl(f_0^l(\zeta,\eta)\bigr)^2\frac{P_{l-1}(\zeta,\eta)}{P_{l+1}(\zeta,\eta)}$ and we compute \eqref{Omega2} (omitting $\zeta$ in $\eta^l_j(\zeta)$):
\begin{multline*}
\sum_{l=1}^{N-1}\sum_{j=1}^{m_l}\frac{d\tilde f_0^l(\zeta,\eta_j^l)}{\tilde f_0^l(\zeta,\eta_j^l)}\wedge d\eta_j^l\\= 2\sum_{l=1}^{N-1}\sum_{j=1}^{m_l}\frac{df_0^l(\zeta,\eta_j^l)}{f_0^l(\zeta,\eta_j^l)}\wedge d\eta_j^l+\sum_{l=1}^{N-1}\sum_{j=1}^{m_l} \left(\sum_{i=1}^{m_{l-1}}\frac{d\eta_i^{l-1}}{\eta_i^{l-1}-\eta_j^l}-
\sum_{i=1}^{m_{l+1}}\frac{d\eta_i^{l+1}}{\eta_i^{l+1}-\eta_j^l}\right)\wedge d\eta_j^l\\
=2\sum_{l=1}^{N-1}\sum_{j=1}^{m_l}\frac{df_0^l(\zeta,\eta_j^l)}{f_0^l(\zeta,\eta_j^l)}\wedge d\eta_j^l.
\end{multline*}
Thus, we are in the situation described in Theorem \ref{two}, and the $SU(N)$-monopole metrics arise from the GLT applied 
to the function 
\begin{equation} F=\frac{1}{2}\left(\oint_c\frac{\eta}{\zeta^2}d\zeta-\frac{1}{2\pi i}\sum_{l=1}^{N-1} (\mu_{l+1}-\mu_l)\oint_{\tilde{0}_l}\frac{\eta^2}{\zeta^3}d\zeta\right)\label{FSU(N)}\end{equation}
on $\sS(m_1,\dots,m_{N-1})$.
The cycle $c$ satisfies $\tau_\ast c=-c$ and it enters each $S_l$ at points of $S_{l,l+1}+S_{l,l-1}$ and leaves at points of 
$S_{l+1,l}+S_{l-1,l}$ ($c$ is determined by the sections $\nu_l$ as in the proof of Theorem \ref{two}).

\begin{remark} For $N=3$ and $\mu_3-\mu_2=\mu_2-\mu_1=1$, the function $F$ is just half of the corresponding to the asymptotic $SU(2)$-monopole metric. Nevertheless, the cycles, and hence the metrics are different. What happens is that the twistor space is the same in both cases, but the real sections corresponding to $SU(3)$-monopoles belong to a different connected component from  the sections corresponding to the asymptotic $SU(2)$-monopole metric. This can be seen from the corresponding Nahm flow, which has a singularity at $-1,0,1$ for the $SU(3)$-monopoles \cite{HuMu}, but is smooth on $(-2,0)$ and on $(0,2)$ for the asymptotic  $SU(2)$-monopole metric. The point is that the triviality of \eqref{L^2mu} does not imply the triviality of \eqref{L^mu}.\end{remark}

Once again, we can guess the form of the asymptotic metric. In the region, where where the monopole of type $l$, $l=1,\dots N-1$, is approximately a superposition of $k$ monopoles of charges $n_1,\dots,n_k$, the asymptotic metric is given by the GLT applied to \eqref{FSU(N)} on $\sS(m_1,\dots,m_{l-1},n_1,\dots,n_k,m_{l+1},\dots,m_{N-1})$.

\subsection{Adjoint orbits and related metrics}

It is by now well-known that adjoint orbits of complex semisimple Lie groups carry hyperk\"ahler metrics (cf. \cite{Kr}). For regular semisimple orbits, the most general construction is due to Alekseevsky and Graev \cite{AG} and to Santa-Cruz \cite{SC}, who associate a $U(k)$-invariant pseudo-hyperk\"ahler structure to any reduced spectral curve $S\in |\sO(2k)|$ (provided $S$ satisfies a reality condition). Suppose that $S$ is a curve given by \eqref{S} and the polynomial coefficients $a_i(\zeta)$ satisfy \eqref{sigma}. The twistor space is defined as 
\begin{equation} Z_S= \left\{ p\in \sO(2)\otimes \gl(k,\cx);\enskip \text{$p$ is a regular matrix}\right\} \label{Z_S} 
\end{equation}
A real section of $Z_S\rightarrow \oP^1$ is a quadratic polynomial $ 
A(\zeta)=A_0+A_1\zeta+A_2\zeta^2$, such that $A(\zeta)$ is a regular matrix for every $\zeta$ and 
which satisfies $ A_0=-A_2^\ast, \enskip A_1=A_1^\ast $. Such a real section is a twistor line if, in addition, the normal bundle of $A(\zeta)$ is the sum of $\sO(1)$'s. This last condition translates into a condition on centralisers of $A(\zeta)$ - see \cite[Theorem 4]{SC}.
The manifold $N_S$ of twistor lines is a pseudo-hyperk\"ahler manifold.
Observe, that a fibre of $Z_S$ over a  $\zeta\in \oP^1$, such that the fibre of $S$ over it consists of distinct points, is an adjoint $GL(k,\cx)$ orbit. The twisted form $\Omega$ is on such a fibre just the Kostant-Kirillov-Souriau.  Consequently, with respect to the complex structure corresponding to such a (generic) $\zeta\in \oP^1$, $N_S$ is isomorphic to an open subset of an adjoint orbit. The well-known complete hyperk\"ahler metrics of Kronheimer correspond to $S$ fully reducible, i.e. a union of rational curves.
\par
We now claim that the pseudo-hyperk\"ahler structure of $N_S$ can be obtained via the generalised Legendre transform.
We consider the space $W_S\subset \sS(1,\dots,k)$, defined by setting $S_k=S$, and apply the GLT to the function 
\begin{equation} \oint_c \frac{\eta}{\zeta^2}d\zeta\label{Forbit}\end{equation}
on $W_S$. Indeed, the results of \cite{Bie-Pidst} imply that the Kostant-Kirillov-Souriau form of regular adjoint orbits of $GL(k,\cx)$ is trivialised in coordinates given by the Gelfand-Zeitlin map (considered in \S \ref{conj}) and by the Gelfand-Zeitlin torus. Thus, the twistor space of $N_S$ can be trivialised, owing to Proposition \ref{O(2)} with $d=2$, by the curves $S_l$ and sections $\sigma_l$ of $\sO(2)[D_{l-1}-D_l]$ satisfying $\sigma_l\overline{\tau^\ast \sigma_l}=\frac{P_{l+1}}{P_{l-1}}$, where $P_l=P_l(\zeta,\eta)$ is the polynomial defining $S_l$. We now proceed as for $SU(N)$-monopoles and conclude, from Theorem \ref{two}, that the metric of $N_S$ is given by the function $\eqref{Forbit}$ on $W_S$.

\medskip

We can also consider the function \eqref{Forbit} on the full $\sS(1,\dots,k)$. We obtain a (pseudo)-hyperk\"ahler manifold $N$, from which all the $N_S$ can be produced via the twistor quotient construction, as in \S\ref{gen}. The complex symplectic structure of $N$ is that of $GL(k,\cx)\times P$, where $P\simeq \cx^k$ is a regular Slodowy slice (cf. \cite{Bie-Pidst}). The metric on $N$ is a limiting case of the metrics on moduli spaces of $SU(k+1)$-monopoles of charge $(1,\dots,k)$, if we allow $\mu_p-\mu_{p+1}\rightarrow 0$ for $p=1,\dots,k$. The metric on $N$ probably has an $SU(N)$-symmetry, just like the metrics on each $N_S$.

\section{Hyperk\"ahler metrics corresponding to $[\eta^2/\zeta^2]$}
 
Theorem \ref{one} implies  that there should be a whole hierarchy of hyperk\"ahler manifolds analogous to $SU(2)$-monopole spaces and corresponding to other $H$ in the formula \ref{Fbundle} (with $k=1$). Their twistor spaces are obtained by glueing two copies of the space of rational maps of degree $n=m_1$ as in \cite[pp. 49--50]{AH}. The real sections correspond to 
spectral curves on which the line bundle  with transition function $\exp\frac{\partial H}{\partial \eta}$ is trivial. Let us write $l(\zeta,\eta)=\frac{\partial H}{\partial \eta}$ and $E^s$ for the line bundle with the transition function $\exp sl(\zeta,\eta)$. From the description of $H^1(S,\sO_S)$,  we have
$$ l(\zeta,\eta)=\sum_{i=1}^{n-1}\frac{\eta^i}{\zeta^i}q_i(\zeta),\quad q_i(\zeta)=\sum_{r=-i+1}^{i-1} d_{r,i}\zeta^r,$$
for some complex numbers $d_{r,i}$. Moreover, $E=E^1$ is real and $H$ satisfies   \eqref{Greal} if and only if $\overline{d_{r,i}}=(-1)^rd_{-r,i}$. Now, according to the general theory \cite{AHH} the flow in the direction $E^s$ on the affine Jacobian $J^{g-1}_\text{aff}$ of line bundles of degree $g-1$ corresponds to a flow of matricial polynomials, and, hence, a hyperk\"ahler metric exists on the space of matricial flows which correspond to periodic flows. The periodicity means that the matrices should have the same behaviour at $s=1$ as at $s=0$; the latter being canonically determined by the flow on the Jacobian approaching the bundle $\sO_S(n-2)\in   J^{g-1}$. 
\par
In the simplest case, $l(\zeta,\eta)=\frac{\eta}{\zeta}$, one obtains Nahm's equations. We wish to discuss briefly the next simplest case $l(\zeta,\eta)=\frac{\eta^2}{\zeta^2}$. One obtains a flow of endomorphisms $\tilde{A}(s,\zeta)=\tilde{A}_0(s)+\tilde{A}_1(s)\zeta+\tilde{A}_2(s)\zeta^2$ of the vector space $H^0\bigl(S, E^s(n-1)\bigr)$  by the general prescription as in \cite{AHH} or \cite{Hit}.
To obtain matrices $A(s,\zeta)=A_0(s)+A_1(s)\zeta+A_2(s)\zeta^2$, one needs to choose a connection. Since we want the matrices to satisfy the Hermitian conditions 
\begin{equation} A_0(s)^\ast=-A_2(s),\enskip A_1(s)^\ast=A_1(s),\label{herm}\end{equation} we choose the connection which preserves the Hitchin metric \cite[eq. (6.1)]{Hit} on $H^0\bigl(S, E^s(n-1)\bigr)$. By analogy with \cite{Hit} one considers $\frac{\tilde{A}(s,\zeta)^2}{\zeta^2}$ and takes half of the $\zeta$- constant term together with the positive terms. One can check, as in \cite[pp. 179-181]{Hit}, that this connection 
\begin{equation} \nabla_sf=\frac{\partial f}{\partial s}+\left(\frac{1}{2}\bigl(\tilde{A}_1^2+\tilde{A}_0\tilde{A}_2+\tilde{A}_2\tilde{A}_0\bigr)+ \bigl(\tilde{A}_1\tilde{A}_2+\tilde{A}_2\tilde{A}_1\bigr)\zeta + \tilde{A}_2^2\zeta^2\right)f\label{connection}\end{equation}
preserves the metric and gives, after some manipulation, the following equations on the matrices $A_i(s)$:
\begin{eqnarray*}\frac{\partial A_0}{\partial s} & = &\frac{1}{2}[A_0,A_1^2]+\frac{1}{2}[A_0^2, A_2]\nonumber \\
\frac{\partial A_2 }{\partial s} & = &\frac{1}{2}[A_0,A_2^2]+\frac{1}{2}[A_1^2, A_2]\\
\frac{\partial A_1 }{\partial s} & = & A_0A_1A_2-A_2A_1A_0+\frac{1}{2}A_0A_2A_1-\frac{1}{2}A_1A_2A_0+\frac{1}{2}A_1A_0A_2-\frac{1}{2}A_2A_0A_1. \nonumber
\end{eqnarray*}
These equations are invariant under the real structure \eqref{herm} and, if we set $A_0=T_2+iT_3, \enskip A_1=i\sqrt{3}T_1,\enskip A_2=T_2-iT_3$, we obtain the following system of ODE's for the $n\times n$  skew-hermitian matrices  $T_1,T_2,T_3$:
\begin{eqnarray*} \frac{\partial T_1}{\partial s} & = &i\bigl(T_3T_1T_2-T_2T_1T_3+T_3T_2T_1-T_1T_2T_3+T_1T_3T_2-T_2T_3T_1\bigr)\\
\frac{\partial T_2}{\partial s} & = & \frac{3}{2}i\bigl[T_3,T_2^2-T_1^2\bigr]\\
\frac{\partial T_3}{\partial s} & = & \frac{3}{2}i\bigl[T_3^2-T_1^2,T_2\bigr].\end{eqnarray*}
We do not know whether these equations have occurred in a different context. The correct boundary conditions guaranteeing perodicity of the flow $E^s$ need investigating (along lines of \cite{Hit}). We expect that $s=0$ and $s=1$ are  regular singular points for the $T_i$, which have there an expansion $(s-a)^{-1/2}\cdot(\text{\em analytic})$, $a=0,1$. In addition, the leading term will depend on the coefficients of the curve, unlike for Nahm's equations. The metric will not be an $L^2$-metric; it has to be computed from the formula \eqref{Omega2}. We also note that for $n=1$ we obtain the metric considered in Example \ref{harm}.

\end{document}